\documentclass{article}
\usepackage{mystyle}

\crefname{equation}{}{}

\newcommand{\customref}[2]{\textup{(}\hyperref[#1]{\textup{#2}}\textup{)}}

\title{A convex lifting approach for the Calder\'on problem}
\date{July 1, 2025}

\author{Giovanni S.~Alberti\footnotemark[1]\thanks{Machine Learning Genoa Center (MaLGa), Department of Mathematics, Department of Excellence 2023–2027, University of Genoa, Italy (giovanni.alberti@unige.it, simone.sanna@edu.unige.it).} \and Romain Petit\footnotemark[2]\thanks{CNRS and DMA, ENS, PSL Universit\'e, France (romain.petit@ens.fr).} \and Simone Sanna\footnotemark[1]}

\begin{document}
\maketitle
\begin{abstract}
The Calder\'on problem consists in recovering an unknown coefficient of a partial differential equation from boundary measurements of its solution. These measurements give rise to a highly
nonlinear forward operator. As a consequence, the development of reconstruction methods for this inverse problem is challenging, as they usually suffer from the problem of local convergence. To circumvent this issue, we propose an alternative approach based on lifting and convex relaxation techniques, that have been
successfully developed for solving finite-dimensional quadratic inverse problems. This leads to a convex optimization problem whose solution coincides with the sought-after coefficient, provided that a non-degenerate source condition holds. We demonstrate the validity of our approach on a toy model where the solution of the partial differential equation is known everywhere in the domain. In this simplified setting, we verify that the non-degenerate source condition holds under certain assumptions on the unknown coefficient. We leave the investigation of its validity in the Calder\'on setting for future works.
\end{abstract}

{\bf Keywords:}
inverse problems for PDEs, Calderón problem, lifting, convex relaxation, nuclear norm minimization.\\

\pagenumbering{arabic}

\section{Introduction}

\subsection{Electrical impendance tomography and the Calder\'on problem}\label{sub:eit_calderon}
In many situations of practical interest in imaging sciences, one wishes to estimate some physical quantity inside a medium without being able to measure it directly. This is the case in electrical impedance tomography, where one aims at reconstructing the electrical conductivity of a medium from boundary voltage-current measurements. This reconstruction is especially interesting in medical imaging, where different types of biological tissues can be discriminated by their electrical conductivity.

The formalization of the inverse conductivity problem is due to \citet{calderon_inverse_1980} (see also \cite{uhlmannElectricalImpedanceTomography2009,feldman2019calderon}). Its unknown is modeled by a bounded positive function $\sigma\colon\Omega\to\RR$ on a domain $\Omega$. Applying an electric potential $f$ on the boundary of the domain induces a potential $u$ in $\Omega$ that solves the following conductivity equation
\begin{equation}
    \left\{\begin{aligned}
    &\mathrm{div}(\sigma\nabla u)=0&&\mathrm{in}~\Omega,\\
    &u=f&&\mathrm{on}~\partial\Omega.
    \end{aligned}\right.
    \label{eq_conductivity}
\end{equation}
The electrical current induced at the boundary is then given by $\restriction{\sigma\partial_{\nu}u}{\partial\Omega}$, where $\nu$ denotes the outer unit normal to $\Omega$. The inverse conductivity problem consists in recovering $\sigma$ from the voltage-to-current map (also called Dirichlet-to-Neumann map) $\Lambda_{\sigma}\colon f\mapsto\restriction{\sigma\partial_{\nu}u}{\partial\Omega}$. 

In this work, we focus on a slight modification of the Calder\'on problem. It consists in recovering an unknown coefficient $q$ from the Dirichlet-to-Neumann map $\Lambda_q\colon f\mapsto\restriction{\partial_{\nu}u}{\partial\Omega}$ associated to the following stationary Schrödinger partial differential equation (PDE)
\begin{equation}\left\{
    \begin{aligned}
        -\Delta u +q u&=0&&\mathrm{in}~\Omega,\\
        u&=f&&\mathrm{on}~\partial\Omega.
    \end{aligned}\right.
    \label{eq_gelfand_calderon}
\end{equation}
This problem can be seen as a reduction of the Calder\'on problem through the change of variable $q=\Delta\sqrt{\sigma}/\sqrt{\sigma}$ (see for instance \cite{sylvester1987global}).

Due to the presence of the product $qu$ in \eqref{eq_gelfand_calderon}, the forward map $q\mapsto \Lambda_q$ is \emph{highly nonlinear}. This is a major source of difficulty for solving the inverse problem. It typically leads to reconstruction methods that suffer from local convergence: they are only guaranteed to recover the unknown when initialized in a small neighborhood around it. Another major difficulty stems from the \emph{ill-posedness} of the inverse problem. Since only boundary measurements are available (the Neumann data associated to the Dirichlet data $f$), variations of $q$ away from $\partial\Omega$ result in very small variations of $\Lambda_{q}$.

In this work, we explore a novel approach to the Calder\'on problem with the aim of addressing the issue of local convergence. This approach is based on convex lifting techniques that were introduced for solving quadratic inverse problems with finite-dimensional unknowns. These methods overcome the nonlinearity of the forward map by lifting the original problem to a higher dimensional space. Under suitable assumptions, they lead to a convex optimization problem whose unique solution is the sought-after unknown, which yields globally convergent reconstruction algorithms. Although we do not address the ill-posedness issue in this work, our approach seems amenable to regularization (as in \citet{liSparseSignalRecovery2013,mcraeOptimalConvexLifted2023}), which could lead to improved stability properties.

\subsection{Previous works}
Both formulations of the Calder\'on problem have been extensively studied, with a primary focus on the injectivity of the mappings $\sigma\mapsto\Lambda_\sigma$ and $q\mapsto\Lambda_q$, which ensures the identifiability of the unknown (see for instance \cite{kohn1984determining,sylvester1987global,astala2006calderon,Bukhgeim+2008+19+33,Isakov2017-qr,feldman2019calderon,uhlmannElectricalImpedanceTomography2009,caro2016global}). 
Beyond identifiability, a lot of attention was dedicated to the study of the well-posedness of the problem and to the development of reconstruction algorithms. 

Regarding stability, both problems are known to be severely ill-posed. A weak stability result (namely logarithmic stability) was proved in \cite{alessandrini1988stable} (see also \citet{novikov2010global,santacesaria2013new,Santacesaria+2015+51+73} for the Schr\"odinger equation case in two dimensions) and later shown to be optimal in \cite{mandache2001exponential}. To obtain improved stability estimates, one can make a priori assumptions on the unknown, such as assuming that it belongs to a known finite dimensional subspace. In this setting, stability estimates assuming the availability of infinitely many measurements were obtained for the first time in  \cite{alessandrini2005lipschitz} (see also \cite{bacchelli2006lipschitz, beretta2013lipschitz,beretta2022global} and references therein). More recently, the case of finitely many measurements was investigated in \cite{harrach2019uniqueness,alberti2022infinite,alberti2022inverse}. 

Concerning reconstruction algorithms, nonlinearity makes the convergence of iterative methods heavily dependent on their initialization, see e.g.\ \cite{kaltenbacher-neubauer-scherzer-2008,kindermann-2022,kaltenbacher-2024}. As a result, the iterative methods studied in the aforementioned works are only guaranteed to converge to the unknown under assumptions that are difficult to satisfy in practice (see \cite{lazzaro-morigi-ratti-2024} for a notable exception). In addition to iterative methods, a class of methods called ``D-bar methods'', based on the theoretical works  \cite{nachman1988reconstructions,nachman1996global}, have been proposed in \cite{siltanen2000implementation,knudsen2009regularized,mueller2020d}. Let us also mention a reversion technique based on the Taylor series of the forward map \citep{garde2023series} and a direct method for partial boundary data \citep{hauptmann2017direct}. 

Over the past years, resarch has been devoted to investigate the possibility of expressing the reconstruction problem as a convex optimization problem, as it is common in inverse problems (see e.g.\ \cite{candes2013phaselift,bubba-burger-helin-ratti-2023}). For the convexification of coefficient inverse problems, Carleman estimates have been shown to play a crucial role \citep{klibanov2004carleman,klibanov2013carleman,bellassoued2017carleman}. They lead to the resolution of convex optimization problems and yield globally convergent reconstruction algorithms (see, for instance, \cite{klibanov2017globally,klibanovConvexificationElectricalImpedance2019}). Recently, another reconstruction method based on the resolution of a convex optimization problem has been proposed in \cite{harrachCalderonProblemFinitely2023}. Its remarkable result shows that the sought-after coefficient is the unique solution of a semidefinite programming problem, and is based on monotonicity estimates for the conductivity equation. However, this approach seems numerically challenging \cite{alberti-petit-poon-2025}.

\subsection{Contributions}\label{subsection problem formulation and contributions}
In this work, we introduce a convex lifting approach for solving the Calder\'on problem. Our core idea is to introduce a new unknown $F(x,y)=u(x)q(y)$ and to consider the Poisson equation
\begin{equation*}\label{eq Schrodinger equation}
    \Delta u(x)=F(x,x), \qquad x\in\Omega.
\end{equation*}
With this change of variable, the link between the unknown $F$ and the solution $u$ becomes linear, while the link between the original unknown $q$ and the solution of \Cref{eq_gelfand_calderon} is highly nonlinear. As a result, we recast the recovery of $q$ as the recovery of a rank-one function of $2d$ variables satisfying a set of linear constraints. Then, by using the nuclear norm as a convex proxy for the rank, we propose to produce an estimate of the unknown by solving a convex variational problem depending only on the available measurements.

To prove that this procedure succeeds, we show that nuclear norm minimization techniques, that were developed for solving quadratic inverse problems with finite-dimensional unknowns, can be applied to the recovery of rank-one linear operators between Hilbert spaces. We show that, if a \emph{non-degenerate source condition} holds, the unique solution to the nuclear norm minimization problem is the sought-after coefficient. We also provide robustness guarantees in the presence of measurement noise. We prove that, for a toy inverse problem where internal measurements are available, under suitable assumptions on the unknown, the non-degenerate source condition indeed holds. The investigation of its validity in the case of more general unknowns or boundary measurements is left for future works.

To our knowledge, this approach for solving coefficient inverse problems for partial differential equations is new. It has the potential to provide globally convergent reconstruction algorithms and is amenable to regularization. We also stress that we are not aware of other works on nuclear norm minimization techniques in an infinite-dimensional setting.

\paragraph{Plan of the paper.} In \Cref{sec_nuclear_norm_min}, we provide a brief introduction to nuclear norm minimization techniques for solving quadratic inverse problems with finite-dimensional unknowns. In \Cref{sec_nuc_norm_min_for_rank-one_op,sec_bochner}, we collect some results on the recovery of rank-one linear operators via nuclear norm minimization and on Bochner spaces. In \Cref{sec_calderon}, we introduce our convex lifting approach for solving the Calder\'on problem. We show that a non-degenerate source condition is the key ingredient for proving exact and robust recovery results. Owing to the complexity of the forward operator associated to the Calder\'on problem, proving that this condition holds is challenging. In \Cref{sec_internal_measurements}, we apply our convex lifting approach to a simpler inverse problem where internal measurements are available. In this setting, we show that the non-degenerate source condition holds under suitable assumptions on the unknown. Finally, in \Cref{sec:conclusion} we provide a short summary of the paper and discuss possible future perspectives.

\section{Nuclear norm minimization for quadratic inverse problems in finite dimension}
\label{sec_nuclear_norm_min}
In this section, we review a successful approach to solve a specific class of nonlinear inverse problems, namely quadratic inverse problems. The main interest of this approach is that it does not suffer from the problem of local convergence. It consists in finding a convex optimization problem whose unique solution is the original unknown. Since iterative algorithms for solving convex problems converge to a global minimizer independently of their starting point, this strategy provides a globally convergent reconstruction method.

\subsection{Quadratic inverse problems}
Quadratic inverse problems are arguably the most elementary nonlinear inverse problems. In the simplest case, they consist in recovering an unknown vector $x^\dagger\in\mathbb{K}^n$ ($\mathbb{K}$ being either $\RR$ or $\CC$) from observations $\m=(\m_k)_{1\leq k\leq m}$ of the form $\m_k=\langle V_k x^\dagger,x^\dagger\rangle$ where $V_k$ is a symmetric ($\KK=\RR$) or hermitian ($\KK=\CC$) matrix of size $n\times n$. These problems are indeed nonlinear, as the mapping from $x^\dagger$ to $\m$ is quadratic.

An example of such a problem is phase retrieval \citep{shechtmanPhaseRetrievalApplication2015}. In this case the matrix $V_k$ is chosen to be $V_k=v_k v_k^*$ where $v_k\in\KK^n$ and $v_k^*$ is the Hermitian transpose of $v_k$. This choice yields $\m_k=|\langle v_k,x^\dagger\rangle|^2$. This problem arises from applications in optics, where light is modeled by a complex-valued wave. In this context, recording the phase of the signal is difficult and observations usually consist of intensity measurements only. This leads to the reconstruction of $x^\dagger$ from measurements of the form $\m_k=|\langle v_k,x^\dagger\rangle|^2$ where $(v_k)_{1\leq k\leq m}$ is a set of sensing vectors.

In the following, we only discuss the case $\KK=\RR$ for simplicity.

\subsection{Lifting and convex relaxation}
\label{subsec_lifting_relaxation}
Lifting methods for solving quadratic inverse problems were introduced in \cite{chaiArrayImagingUsing2010,candes2013phaselift}. The fundamental observation is that, given $x\in\RR^n$, the quantity $\langle V_k x,x\rangle$ is in fact equal to the Frobenius inner product $\langle V_k,X\rangle$ of $V_k$ and $X\eqdef xx^T$. While its dependence on $x$ is quadratic -- hence nonlinear -- its dependence on $X$ is \emph{linear}. This leads to the observation that the following two problems are equivalent:
	\begin{align}
	&\mathrm{find}~x\in\RR^n &&\text{such that}~~~\text{for every}~k\in\{1,...,m\},~\langle V_k x,x\rangle=\m_k,\label{prob}\\
	&\mathrm{find}~X\in \mathcal{S}_n(\RR)&&\text{such that}~~~X\succeq 0,~\mathrm{rank}(X)=1~~\mathrm{and}~~\text{for every}~k\in\{1,...,m\},~\langle V_k,X\rangle=\m_k,\label{prob_lift}
	\end{align}
where $\mathcal{S}_n(\RR)$ is the set of symmetric $n\times n$ real matrices and $X\succeq 0$ means that $X$ is positive semi-definite. Going from \cref{prob} to \cref{prob_lift} allows us to linearize the constraint on the unknown. This comes at the price of increasing the number of variables and introducing a non-convex rank one constraint.

The key idea introduced in \cite{candes2013phaselift} is to relax the rank one constraint and minimize instead a natural convex proxy for the rank, namely the nuclear norm. The nuclear norm of a matrix $X\in\RR^{n\times n}$, denoted by $\|X\|_*$, is defined as the sum of its singular values. In the case of symmetric semi-definite matrices, it coincides with the trace. As a result, the above-mentioned authors proposed to consider the following optimization problem
\begin{equation}
    \underset{X\in \mathcal{S}_n(\RR)}{\mathrm{min}}~\|X\|_*~~\mathrm{s.t.}~~X\succeq 0~~\mathrm{and}~~\text{for every}~k\in\{1,...,m\},~\langle V_k,X\rangle=\m_k.
    \label{prob_relaxed_tmp}
\end{equation}
This problem is convex and hence can  be solved by iterative algorithms with global convergence guarantees. It is in fact a semi-definite programming problem for which efficient solvers exist. When only a noisy version~$\m^\delta$ of $\m$ is available, it is natural to consider
\begin{equation}
    \underset{X\in \mathcal{S}_n(\RR)}{\mathrm{min}}~\frac{1}{2}\sum\limits_{k=1}^m (\langle V_k,X\rangle-\m^\delta_k)^2+\lambda \|X\|_*~~\mathrm{s.t.}~~X\succeq 0,
    \label{prob_relaxed_noisy_tmp}
\end{equation}
where $\lambda >0$ is a regularization parameter. Two natural questions regarding this convex relaxations arise. Is $X^\dagger=x^\dagger(x^\dagger)^T$ the unique solution to \cref{prob_relaxed_tmp}? Is $X^\dagger=x^\dagger(x^\dagger)^T$ close to every solution to \cref{prob_relaxed_noisy_tmp} when $\|\m^\delta-\m\|$ is small? In the next subsection, we review sufficient conditions that ensure a positive answer to both questions. 

In fact, to be closer to the setting of \Cref{sec_internal_measurements,sec_calderon}, we study a slightly different problem in which the lifted unknown could possibly be non-square and non-symmetric. Namely, we consider the unknown $X^\dagger=\sigma uv^T$ where $u\in\RR^{n_1}$ and $v\in\RR^{n_2}$ are unit vectors and $\sigma>0$. Given a linear measurement operator $\Phi\colon\RR^{n_1\times n_2}\to\RR^m$, we define $\m=\Phi X^\dagger$ and wish to recover $X^\dagger$ by solving
\begin{equation}
    \underset{X\in\RR^{n_1\times n_2}}{\mathrm{min}}~\|X\|_*~~\mathrm{s.t.}~~\Phi X=\m.
    \tag{$\mathcal{P}$}
    \label{primal}
\end{equation}
When only a noisy version $\m^\delta$ of $\m$ is available, we solve instead
\begin{equation}
    \underset{X\in\RR^{n_1\times n_2}}{\mathrm{min}}~\frac{1}{2}\|\Phi X-\m^\delta\|^2+\lambda\|X\|_*
    \tag{$\mathcal{P}^\delta_{\lambda}$}
    \label{primal_noisy}
\end{equation}
with $\lambda >0$. In the next subsection, we review sufficient conditions under which we have the following two properties.
\begin{enumerate}
    \item Exact recovery: the unique solution to \cref{primal} is $X^\dagger$.
    \item Robust recovery: every solution to \cref{primal_noisy} is close to $X^\dagger$ when $\m^\delta$ is close to $\m$.
\end{enumerate}
As we provide the proofs of these results in the more general setting of recovering rank-one linear operators between Hilbert spaces in \Cref{sec_nuc_norm_min_for_rank-one_op}, we state all the results of the next subsection without proof. We refer the reader to \citet[Section 3]{Candes2009-fw} and \citet[Section 5]{vaiterLowComplexityRegularizations2014} for more details on duality arguments and exact recovery results, and to \citet[Section 6]{vaiterLowComplexityRegularizations2014} for noise robustness results.

\subsection{Exact and robust recovery}
\label{subsec_recovery_nuc_norm_min}
As it is often the case in convex optimization, relevant information about the solutions of \cref{primal} can be obtained by considering its dual problem. This problem is given by 
\begin{equation}
    \underset{p\in \RR^m}{\mathrm{sup}}~\langle p,\m\rangle~~\mathrm{s.t.}~~\|\Phi^*p\|\leq 1,
    \tag{$\mathcal{D}$}
    \label{dual}
\end{equation}
where $\|\cdot\|$ denotes the operator norm of matrices, equal to the largest singular value. This norm is dual to the nuclear norm, that is to say
\begin{equation*}
    \|X\|_*=\mathrm{sup}\,\{\langle X,Y\rangle\,\rvert\,Y\in\RR^{n_1\times n_2},~\|Y\|\leq 1\},\qquad X\in\RR^{n_1\times n_2}.
\end{equation*}

Strong duality holds between \cref{primal} and \cref{dual} in the sense of the following proposition.
\begin{proposition}\label{proposition strong duality finite dimension}
    Both \cref{primal} and \cref{dual} have solutions and their values are equal. Moreover, if $p$ is a solution to \cref{dual} and $H=\Phi^*p$, then every solution $X$ to \cref{primal} satisfies 
    \begin{equation}
        \langle X,H\rangle=\|X\|_*.
        \label{opt_primal_dual}
    \end{equation}
    Conversely, if $\Phi X=\m$, \cref{opt_primal_dual} holds with $H=\Phi^*p$ and $\|H\|\leq 1$ then $X$ and $p$ are respectively solutions to \cref{primal} and \cref{dual}.
\end{proposition}

As a result, to ensure that $X^\dagger$ is a solution to \cref{primal}, it is sufficient to show the existence of $H\in\mathrm{Im}(\Phi^*)$ such that $\|H\|\leq 1$ and $\langle X^\dagger,H\rangle =\|X^\dagger\|_*$. This property is known as \emph{source condition} (see e.g.\ \cite{engl-hanke-neubauer-1996,benning-bubba-ratti-riccio-2024}). If such an $H$ exists, it is called a \emph{dual certificate}, as it is directly linked to a solution of the dual problem \cref{dual} and guarantees the optimality of $X^\dagger$ for \cref{primal}. In fact, the set $\{H\in\RR^{n_1\times n_2}\,\rvert\, \|H\|\leq 1,~\langle X,H\rangle=\|X\|_*\}$ is nothing but the subdifferential of the nuclear norm at $X$, denoted by $\partial\|\cdot\|_*(X)$ (see \citet[Section 23]{rockafellarConvexAnalysis2015} for a precise definition and \cite{watsonCharacterizationSubdifferentialMatrix1992} for the special case of the nuclear norm).

As we are interested in certifying the optimality of $X^\dagger=\sigma uv^T$ for \cref{primal}, we give another useful characterization of $\partial\|\cdot\|_*(X^\dagger)$  below.
\begin{lemma}
    Let $H\in\RR^{n_1\times n_2}$. The following conditions are equivalent.
    \begin{enumerate}
        \item It holds that $\|H\|\leq 1$ and $\langle X^\dagger,H\rangle=\|X^\dagger\|_*$.
        \item There exists $W\in\RR^{n_1\times n_2}$ such that $H=uv^T+W$ with $Wv=0$, $W^Tu=0$ and $\|W\|\leq 1$.
    \end{enumerate}
    \label{lemma_equiv_subdiff}
\end{lemma}
The source condition (the existence of a dual certificate for $X^\dagger$) only ensures that $X^\dagger$ is \emph{a} solution to \cref{primal}. If one is interested in $X^\dagger$ being the \emph{unique} solution to \cref{primal}, the source condition can be strengthened to require the existence of a \emph{non-degenerate} dual certificate.
\begin{definition}
    We say that the non-degenerate source condition holds if
    \begin{equation}
        \mathrm{there~exists}~H\in\mathrm{Im}(\Phi^*)~\mathrm{s.t.}~H=uv^T+W~~\mathrm{with}~Wv=0,~W^Tu=0~\mathrm{and}~\|W\|<1.
        \label{ndsc finite}
        \tag{NDSC}
    \end{equation}
\end{definition}
The non-degenerate source condition simply amounts to requiring a \emph{strict} inequality in condition 2.\ of \Cref{lemma_equiv_subdiff}. Its main interest is that it allows us to obtain exact recovery guarantees, as stated in the following proposition.
\begin{proposition}
    Assume that $\m\neq 0$ and that the non-degenerate source condition holds. Then $X^\dagger=\sigma uv^T$ is the unique solution to~\cref{primal}.
    \label{prop_exact_recovery}
\end{proposition}
Let us stress that the assumptions of \Cref{prop_exact_recovery} are sufficient but not necessary for exact recovery. A sufficient and necessary condition can be found in \citet[Corollary 4.8]{fadiliSolutionUniquenessConvex2024}. As we have not been able to leverage this weaker type of condition in \Cref{sec_internal_measurements,sec_calderon}, we do not discuss it here.

To obtain robust recovery guarantees in the presence of measurement noise, the non-degenerate source condition can be complemented with an injectivity condition on the model tangent subspace $T$ defined by
\begin{equation}\label{eq_T_finite}
T\eqdef \{u\otimes a+b\otimes v:(a,b)\in \RR^{n_2}\times \RR^{n_1}\}.
\end{equation}
This is the object of the following proposition, in which $\|\cdot\|_F$ denotes the Frobenius (or Hilbert-Schmidt) norm of matrices.
\begin{proposition}
    Let $\delta$ and $c$ be two positive constants. If the non-degenerate source condition \cref{ndsc finite} holds and $\Phi$ is injective on $T$ then, for any minimizer $X^\delta$ of \cref{primal_noisy} with $\lambda=c\delta$ and $\m^\delta\in\RR^m$ such that $\|\m^\delta-\m\|\leq \delta$, we have that $\|X^\delta-X^\dagger\|_F=\mathcal{O}(\delta)$.
    \label{prop_robust_recovery}
\end{proposition}
The main weakness of \Cref{prop_robust_recovery} is that it does not show that $X^\delta$ is of the form $\tilde{\sigma}\tilde{u}\tilde{v}^T$ with $\tilde{\sigma}$, $\tilde{u}$ and $\tilde{v}$ respectively close to $\sigma$, $u$ and $v$. This stronger robustness property is linked to the notion of \emph{model selection} (see for example \citet[Section 7]{vaiterLowComplexityRegularizations2014} and \cite{vaiterModelConsistencyPartly2018,fadiliModelConsistencyLearning2019}, as well as \cite{bachConsistencyTraceNorm2008} for the special case of nuclear norm minimization). In order to prove that model selection occurs at low noise, a stronger version of the non-degenerate source condition \cref{ndsc finite}, which involves the \emph{minimal norm dual certificate}, has to be considered. As it does not appear straightforward to compute this special dual certificate in the setting of \Cref{sec_internal_measurements,sec_calderon}, we do not discuss model selection properties here.
\section{Nuclear norm minimization for recovering rank-one operators}\label{sec_nuc_norm_min_for_rank-one_op}
In this section, we extend the approach discussed in the previous section to the problem of recovering rank-one operators between Hilbert spaces.

As in \Cref{sec_nuclear_norm_min}, the nuclear norm plays a crucial role in the convex lifting approach we propose in \Cref{sec_internal_measurements,sec_calderon}. Because we work in an infinite-dimensional setting (our unknowns are functions while the unknowns of \Cref{sec_nuclear_norm_min} are finite-dimensional vectors), we need to rely on the notion of nuclear norm of linear operators.

\subsection{The nuclear norm of linear operators}
 In this subsection, we recall the definition of the nuclear norm of linear operators and collect some useful properties. We refer the reader to \cite{reed1972methods,conway2000course} for more details on this topic.

\paragraph{Definition.} Let $\calH_1,\calH_2$ be separable real Hilbert spaces. We denote  the space of compact linear operators from $\calH_1$ to $\calH_2$ by $\mathcal{B}_0(\calH_1;\calH_2)$. If $G\in\mathcal{B}_0(\calH_1;\calH_2)$, its singular value decomposition (see \citet[Chapter 1]{Simon2005-rp}) gives rise to a non-increasing sequence of non-negative singular values $(\sigma_i(G))_{i\in\NN}$. The nuclear norm of $G$ is defined as $\|G\|_*\eqdef\sum_{i\in\NN} \sigma_i(G)\in\RR_{\geq 0}\cup\{+\infty\}$. We denote  the space of compact operators with finite nuclear norm by $\calB_1(\calH_1;\calH_2)$.

\paragraph{The nuclear norm on $\calB_2(\calH_1;\calH_2)$.} In the following, in order to obtain minimization problems over a Hilbert space, we rather minimize the nuclear norm on the space of compact operators with square-summable singular values, denoted by $\calB_2(\calH_1;\calH_2)$. Equipped with the inner product 
\begin{equation}\label{definition inner product B2}
\langle G,H\rangle \eqdef \sum_{i\in\NN}\langle Ge_i, He_i\rangle_{\calH_2},
\end{equation}
where $\{e_i\}_{i\in\NN}$ is an orthonormal basis for $\calH_1$, it is a Hilbert space and the associated norm is the Hilbert-Schmidt norm $\|\cdot\|^2_{\HS}=\sum_{i\in\NN} \sigma_i(\cdot)^2$. Our choice to minimize the nuclear norm on $\calB_2(\calH_1;\calH_2)$ is justified by the fact that the mapping $\|\cdot\|_*\colon\calB_2(\calH_1;\calH_2)\to\RR\cup\{+\infty\}$ is proper convex and lower semi-continuous. This can be proved by considering the dual formula \eqref{dual_formula_nuclear_norm} below. In the following, unless stated otherwise, we always consider $\|\cdot\|_*$ as a function on $\calB_2(\calH_1;\calH_2)$.

\paragraph{Dual formulation and subdifferential.} As in the finite-dimensional case, a dual formula involving the operator norm $\|\cdot\|$ holds for $\|\cdot\|_*$. Since finite-rank operators are dense in $\calB_0(\calH_1;\calH_2)$, we have that $\calB_2(\calH_1;\calH_2)$ is dense in $\calB_0(\calH_1;\calH_2)$, and so one can show that
\begin{equation}
    \|G\|_*=\mathrm{sup}~\{\langle G,H\rangle\,\rvert\,H\in\mathcal{B}_2(\calH_1;\calH_2),~\|H\|\leq 1\},\qquad G\in\calB_2(\calH_1;\calH_2).
    \label{dual_formula_nuclear_norm}
\end{equation}
As in \Cref{sec_nuclear_norm_min}, the subdifferential of the nuclear norm plays a crucial role in our analysis. Considering \eqref{dual_formula_nuclear_norm}, one can show that it is given by
\begin{equation}\label{expression of subdiff nuc norm at F}
    \partial\|\cdot\|_*(G)=\{H\in\calB_2(\calH_1;\calH_2): \|H\|\leq1, \langle H,G\rangle=\|G\|_*\},\qquad G\in\calB_2(\calH_1;\calH_2).
\end{equation}
In the following, we use the tensor notation to refer to rank-one operators. Namely, given $u\in\calH_1$, $v\in\calH_2$, we denote by $u\otimes v$ the rank-one operator
\begin{equation*}
    u\otimes v:\mathcal{H}_1\ni a\mapsto \langle u,a\rangle_{\calH_1}v\in\calH_2.
\end{equation*}
In the case of a rank-one operator $F=\sigma u\otimes v$, where $\sigma>0$ and $u,v$ are unit vectors in $\calH_1$ and $\calH_2$ respectively, the subdifferential of $\|\cdot\|_*$ at $F$ can be characterized in several ways. To state them, as in \eqref{eq_T_finite}, we introduce the model tangent subspace $T\eqdef\{u\otimes b+a\otimes v:(a,b)\in\calH_1\times\calH_2\}$ and denote the orthogonal projection on $T$ (respectively $T^\perp$) by $P_T$ (respectively $P_{T^\perp}$).
\begin{lemma}\label{lemma characterization subdiff at F}
    Let $H\in\calB_2(\calH_1;\calH_2)$ and $F=\sigma u\otimes v$, where $\sigma>0$ and $u,v$ are unit vectors in $\calH_1$ and $\calH_2$, respectively. The following conditions are equivalent.
    \begin{enumerate}[label=(\roman*)]
        \item It holds that $\|H\|\leq1$ and $\langle H,F\rangle=\|F\|_*$. \label{i}
        \item It holds that $P_T(H)=u\otimes v$ and $\|P_{T^\perp}(H)\|\leq1$. \label{ii}
        \item There exists $W\in\calB_2(\calH_1;\calH_2)$ such that $H=u\otimes v+W$ with $Wu=0$, $W^*v=0$ and $\|W\|\leq1$. \label{iii}
        \item It holds that $Hu=v$, $H^*v=u$ and $|\langle Hu^\perp,v^\perp\rangle_{\calH_2}|\leq1$ for every pair $(u^\perp,v^\perp)\in \calH_1\times\calH_2$ of unit vectors respectively orthogonal to $u$ and $v$. \label{iv}
    \end{enumerate}
\end{lemma}
\begin{proof}
    We start by reviewing useful properties of $P_T$ and $P_{T^\perp}$. For $H\in\calB_2(\calH_1;\calH_2)$, it holds that (see \citet[Equation (3.5)]{Candes2009-fw} for the finite-dimensional case)
    \begin{align}
        P_T(H)&=HP_u+P_vH-P_vHP_u=u\otimes Hu+H^*v\otimes v-\langle H,u\otimes v\rangle u\otimes v,\label{eq expression proj T}\\
        P_{T^\perp}(H)&=P_{v^\perp} HP_{u^\perp},\label{eq expression proj Tperp}
    \end{align}
    where $P_u$ (respectively $P_v$) is the projection onto $\operatorname{Span\{u\}}$ (respectively $\operatorname{Span}\{v\}$), while $P_{u^\perp}$ (respectively $P_{v^\perp}$) is the projection onto $\operatorname{Span\{u\}}^\perp$ (respectively $\operatorname{Span}\{v\}^\perp$). By the non-expansivity of $P_{u^\perp}$ and $P_{v^\perp}$, we notice that
    \begin{equation}\label{eq non expansivity P_Tperp}
        \|P_{T^\perp}H\|\leq\|H\|,\qquad H\in\calB_2(\calH_1;\calH_2).
    \end{equation}
    
    Now, we first prove that \ref{i} implies \ref{ii}. Let $H$ be such that $\|H\|\leq1$ and $\langle H,F\rangle=\|F\|_*$. In particular we have 
    \[
    1=\langle H,u\otimes v\rangle=
    \langle Hu, v\rangle_{\mathcal{H}_2}.
    \]
    Since $\|H\|\leq1$, we get $Hu=v$, $H^*v=u$. Therefore, by \Cref{eq expression proj T} we have $P_T(H)=u\otimes v$ and by \Cref{eq non expansivity P_Tperp} we have $\|P_{T^\perp}H\|\leq1$.
    
    The fact that \ref{ii} implies \ref{iii} follows immediately by setting $W=P_{T^\perp}H$, thanks to \eqref{eq expression proj Tperp}. 

    Let us now prove that \ref{iii} implies \ref{iv}. By $Wu=0$ and $W^*v=0$ we obtain $Hu=v$ and $H^*v=u$. Since $Hu^\perp=Wu^\perp$ for every $u^\perp$ orthogonal to $u$, then we get $|\langle Hu^\perp,v^\perp\rangle_{\calH_2}|\leq1$ for every pair $(u^\perp,v^\perp)\in \calH_1\times\calH_2$ of unit vectors respectively orthogonal to $u$ and $v$.

    Finally, we prove that \ref{iv} implies \ref{i}. The identity $Hu=v$ implies $\langle H,F\rangle=\sigma\langle Hu,v\rangle_{\calH_2}=\sigma=\|F\|_*$.
    Now, let $(a,b)\in\calH_1\times\calH_2$ and write 
    \begin{equation*}
        a=P_ua+P_{u^\perp}a,\qquad b=P_vb+P_{v^\perp}b.
    \end{equation*}
    The assumptions tell us that $H$ acts as isometry between $\operatorname{span}\{u\}$ and $\operatorname{span}\{v\}$, while it is non-expansive between $\operatorname{span}\{u\}^\perp$ and $\operatorname{span}\{v\}^\perp$. Therefore, we have
    \begin{equation*}
       \lvert\langle Ha,b\rangle_{\calH_2}\rvert
        =\lvert\langle HP_ua,P_vb\rangle_{\calH_2}+ \langle HP_{u^\perp}a,P_{v^\perp}b\rangle_{\calH_2}\rvert\leq\|P_ua\|_{\calH_1}\|P_vb\|_{\calH_2}+\|P_{u^\perp}a\|_{\calH_1}\|P_{v^\perp}b\|_{\calH_2}.
    \end{equation*}
    Hence $| \langle Ha,b\rangle_{\calH_2}|\leq \|a\|_{\calH_1}\|b\|_{\calH_2}$, which implies $\|H\|\leq1$.    
\end{proof}

The last result we prove in this subsection is the following technical lemma. We use it to prove the exact recovery result of \Cref{prop_exact_recovery_op}.

\begin{lemma}\label{lemma techical lemma for showing F unique sol}
    Let $u\in \calH_1$, $v\in\calH_2$ be such that $\|u\|_{\calH_1}=\|v\|_{\calH_2}=1$. Let $H\in\calB_2(\calH_1;\calH_2)$ be such that 
    \begin{equation}
        Hu=v, \quad H^*v=u\quad\mathrm{and} \quad |\langle Hu^\perp, v^\perp\rangle_{\calH_2}|<1
        \label{eq_strict_subdiff}
    \end{equation} 
    for every pair of unit vectors $u^\perp$ and $v^\perp$ respectively orthogonal to $u$ and $v$. Then, for every $G\in\calB_1(\calH_1;\calH_2)$ such that $\langle H,G\rangle=\|G\|_*$ there exists $\sigma\geq0$ such that $G=\sigma u\otimes v$.
\end{lemma}
To prove this lemma we need the following result.
\begin{lemma}\label{lemma techical lemma on Hilbert spaces}
    Let $u\in \calH_1$, $v\in\calH_2$ be such that $\|u\|_{\calH_1}=\|v\|_{\calH_2}=1$ and let $H\in\calB(\calH_1;\calH_2)$ be such that \eqref{eq_strict_subdiff} holds. Then, for every pair of unit vectors $(a,b)\in\calH_1\times\calH_2$ we have $\langle Ha,b\rangle_{\calH_2}\leq1$. Moreover, equality holds if and only if $(a,b)=\pm(u,v)$. 
\end{lemma}
\begin{proof}
   Consider two unit vectors $a\in\calH_1$ and $b\in\calH_2$. Decompose these vectors as 
   \begin{equation*}
       a=a_u+a_{u^\perp},\qquad b=b_v+b_{v^\perp},
   \end{equation*}
    where $a_u=\langle a,u\rangle_{\calH_1}u$, $a_{u^\perp}=a-a_u$, $b_v=\langle b,v\rangle_{\calH_2}v$ and $b_{v^\perp}=b-b_v$. We have that
    \begin{equation*}
        \langle Ha,b\rangle_{\calH_2}=\langle a,u\rangle_{\calH_1}\langle b,v\rangle_{\calH_2}+\langle Ha_{u^\perp},b_{v^\perp}\rangle_{\calH_2}.
    \end{equation*}
    
    If $a_{u^\perp}\neq0$ and $b_{v^\perp}\neq0$, setting $u^\perp = \frac{a_{u^\perp}}{\|a_{u^\perp}\|_{\calH_1}}$ and $v^\perp = \frac{b_{v^\perp}}{\|b_{v^\perp}\|_{\calH_2}}$, we have
    \begin{equation*}
        \langle Ha_{u^\perp},b_{v^\perp}\rangle_{\calH_2}= \|a_{u^\perp}\|_{\calH_1}\|b_{v^\perp}\|_{\calH_2} \langle Hu^\perp,v^\perp\rangle_{\calH_2}< \|a_{u^\perp}\|_{\calH_1}\|b_{v^\perp}\|_{\calH_2},
    \end{equation*}
    which, by Cauchy Schwartz inequality, gives
    \begin{equation*}
    \begin{aligned}
        \langle Ha,b\rangle_{\calH_2}
        &<\langle a,u\rangle_{\calH_1}\langle b,v\rangle_{\calH_2}+\|a_{u^\perp}\|_{\calH_1}\|b_{v^\perp}\|_{\calH_2}\\
        &=\|a_u\|_{\calH_1}\|b_v\|_{\calH_2}+\|a_{u^\perp}\|_{\calH_1}\|b_{v^\perp}\|_{\calH_2}\\
        &\leq\|a\|_{\calH_1}\|b\|_{\calH_2}=1.
    \end{aligned}
    \end{equation*}
    As a consequence, if $a_{u^\perp}\neq0$ and $b_{v^\perp}\neq0$, the equality can never hold. 
    
    If $a_{u^\perp}=0$ or $b_{v^\perp}=0$, then
    \begin{equation*}
        \langle Ha,u\rangle_{\calH_2}=\langle a,u\rangle_{\calH_1}\langle b,v\rangle_{\calH_2}\leq1,
    \end{equation*}
    and, by the equality case in Cauchy Schwartz inequality, the equality holds if and only if $(a,b)=\pm(u,v)$.
\end{proof}
\begin{proof}[Proof of \Cref{lemma techical lemma for showing F unique sol}]
 If $G=0$, the statement is obvious. Assume now $G\neq 0$. Suppose that $\langle H,G\rangle=\|G\|_*$. Consider the singular value decomposition of $G$:
  \begin{equation*}
  G=\sum\limits_{i\in\NN}\sigma_i(G)\varphi_i\otimes \psi_i.
  \end{equation*}
Without loss of generality, assume $\sigma_i(G)>0$ for every $i$. Then
  \begin{equation*}
      \sum\limits_{i\in\NN}\sigma_i(G)=\sum\limits_{i\in\NN}\sigma_i(G)\langle H\varphi_i,\psi_i\rangle_{\calH_2}.
\end{equation*}   
Then $\langle H\varphi_i,\psi_i\rangle_{\calH_2}=1$ for every $i$.
Using \Cref{lemma techical lemma on Hilbert spaces}, we obtain $\varphi_i\otimes\psi_i=a\otimes b$ for every $i$. Therefore, $G=\sigma a\otimes b$ for $\sigma=\sum\limits_{i\in\NN}\sigma_i(G)>0$, as required.
\end{proof}

\subsection{Recovering rank-one operators via nuclear norm minimization}
\label{subsec_nuclear_norm_min_op}
In this subsection, we fix $N\in\NN^*$ and generalize the exact and robust recovery results of \Cref{subsec_recovery_nuc_norm_min} to the simultaneous recovery of $N$ rank-one linear operators. We fix an unknown ${F^\dagger=(F_i^\dagger)_{1\leq i\leq N}=(\sigma_i u_i\otimes v_i)_{1\leq i\leq N}}$ where $u_i\in\calH_1$ and $v_i\in\calH_2$ are unit vectors and $\sigma_i>0$ for every $i$. Given a linear measurement operator ${\Phi\colon\calB_2(\calH_1;\calH_2)^N\to\calH}$, where $\calH$ is a real separable Hilbert space, we define $\m\eqdef\Phi F^\dagger$. Depending on whether we have the exact knowledge of $\m$ or of a noisy version $\m^\delta$ of $\m$, we wish to recover $F^\dagger$ by solving one of the following problems
\begin{equation}
    \underset{F\in\calB_2(\calH_1;\calH_2)^N}{\mathrm{min}}~\sum\limits_{i=1}^N\|F_i\|_*~~\mathrm{s.t.}~~\Phi F=\m,
    \tag{$\mathcal{P}$}
    \label{primal_op}
\end{equation}
\begin{equation}
    \underset{F\in\calB_2(\calH_1;\calH_2)^N}{\mathrm{min}}~\frac{1}{2}\|\Phi F-\m^\delta\|_\calH^2+\lambda\sum\limits_{i=1}^N\|F_i\|_*,
    \tag{$\mathcal{P}_{\lambda}^\delta$}
    \label{primal_noisy_op}
\end{equation}
where $\lambda>0$ is a regularization parameter.

The dual problem to \eqref{primal_op} is 
\begin{equation}
    \underset{p\in \mathcal{H}}{\mathrm{sup}}~\langle p,\m\rangle_{\mathcal{H}}~~\mathrm{s.t.}~~\underset{1\leq i\leq N}{\mathrm{max}}\,\|(\Phi^*p)_i\|\leq 1,
    \tag{$\mathcal{D}$}
    \label{dual_op}
\end{equation}
and strong duality holds between \cref{primal_op} and \cref{dual_op} in the sense of the following proposition.
\begin{proposition}
    There exists a solution to \cref{primal_op} and the values of \cref{primal_op} and \cref{dual_op} are equal. Moreover, if $p$ is a solution to \cref{dual_op} and $H=\Phi^*p$, then every solution $F$ to \cref{primal_op} satisfies 
    \begin{equation}
        \langle F_i,H_i\rangle=\|F_i\|_*,\qquad 1\leq i\leq N.
        \label{opt_primal_dual_op}
    \end{equation}
    Conversely, if $\Phi F=\m$, \cref{opt_primal_dual_op} holds with $H=\Phi^*p$ and $\|H_i\|\leq 1$ for every $1\leq i\leq N$ then $F$ and $p$ are respectively solutions to \cref{primal_op} and \cref{dual_op}.
    \label{prop_strong_dual_op}
\end{proposition}
The proof of this proposition, which relies on standard duality results in convex optimization, is given in \Cref{appendix_strong_duality_internal}.

Applying \Cref{lemma characterization subdiff at F} to $F_i^\dagger$ for every $1\leq i\leq N$, we obtain that $\|H_i\|\leq 1$ and \eqref{opt_primal_dual_op} hold if and only if $H_i=u_i\otimes v_i+W_i$ with $W_iu_i=0$, $W^*_iv_i=0$ and $\|W_i\|\leq 1$. As in \Cref{sec_nuclear_norm_min}, the key to obtain exact and robust recovery results is to strengthen these conditions to obtain a non-degenerate source condition, which we define below. 
\begin{definition}
    We say that the non-degenerate source condition holds if
    \begin{equation}
        \begin{aligned}
        &\mathrm{there~exists}~H\in\mathrm{Im}(\Phi^*)~\mathrm{s.t.}~\mathrm{for~every}~1\leq i\leq N,~H_i=u_i\otimes v_i+W_i\\&\mathrm{with}~W_iu_i=0,~W_i^*v_i=0~\mathrm{and}~\|W_i\|<1.
        \end{aligned}
        \label{ndsc_op}
        \tag{NDSC}
    \end{equation}
\end{definition}
As shown in the proof of \Cref{prop_exact_recovery_op} below, assuming \Cref{ndsc_op}  allows us to conclude that every solution to \cref{primal} belongs to a cone $\mathcal{C}$ defined by $\mathcal{C}\eqdef\prod_{1\leq i\leq N}\{\alpha_iu_i\otimes v_i:\alpha_i\ge0\}$. To obtain the exact recovery property, the \Cref{ndsc_op} therefore has to be complemented with the assumption that $\Phi$ is injective on $\mathcal{C}$. In the case $N=1$, this amounts to assuming that $z\neq 0$ as in \Cref{prop_exact_recovery}.
\begin{proposition}
    Assume that the non-degenerate source condition \Cref{ndsc_op} holds and that $\Phi$ is injective on $\mathcal{C}$. Then $F^\dagger=(\sigma_i u_i\otimes v_i)_{1\leq i\leq N}$ is the unique solution to \cref{primal_op}.
    \label{prop_exact_recovery_op}
\end{proposition}
\begin{proof}
By Lemma~\ref{lemma characterization subdiff at F}, it holds that $\|H_i\|\leq 1$ and ${\langle H_i,F^\dagger_i\rangle=\sigma_i=\|F^\dagger_i\|_*}$ for every $i$. Using \Cref{prop_strong_dual_op}, we obtain the optimality of $F^\dagger$ for \cref{primal_op} and the optimality of $p$ for \cref{dual_op}, where $H=\Phi^*p$. Now, let $F$ be a solution to \cref{primal_op}. Since $p$ solves \cref{dual_op} and $H=\Phi^*p$, by \Cref{prop_strong_dual_op}, we know that $\langle H_i,F_i\rangle=\|F_i\|_*$ for every $i$. By \Cref{ndsc_op}, we notice that $H_iu_i=v_i$, $H_i^*v_i=u_i$ and 
$$
\lvert\langle H_i u^\perp,v^\perp\rangle\rvert=\lvert\langle W_i u^\perp,v^\perp\rangle\rvert<1,
$$
for every unit vectors $u^\perp, v^\perp$ respectively orthogonal to $u,v$, where we used $\|W_i\|<1$ for the last inequality. Hence, applying \Cref{lemma techical lemma for showing F unique sol}, we obtain the existence of $\alpha_i\geq 0$ such that $F_i=\alpha_i u_i\otimes v_i$ for every $i$. Finally, the injectivity of $\Phi$ on $\mathcal{C}$ and the fact that $\Phi F=z=\Phi F^\dagger$ allow us to conclude that $F=F^\dagger$.
\end{proof}
As in the finite-dimensional case, to obtain robust recovery guarantees, the non-degenerate source condition should be complemented with an injectivity condition on the model tangent subspace $T$ defined by
\begin{equation}\label{eq:TandTi}
    T\eqdef\prod_{1\leq i\leq N}T_i~~\mathrm{with}~~T_i=\{u_i\otimes a+b\otimes v_i:(a,b)\in\calH_2\times\calH_1\}.
\end{equation}
The main difference with the finite-dimensional case is that, in addition to the injectivity of $\Phi$ on $T$, we also need to ensure that $\restriction{\Phi}{T}$ has \emph{closed range}.
\begin{proposition}
    Let $\delta$ and $c$ be two positive constants. If $\restriction{\Phi}{T}$ is injective and has closed range and the non-degenerate source condition \cref{ndsc_op} holds, for any minimizer $F^\delta$ of \cref{primal_noisy_op} with $\lambda=c\delta$ and $\m^\delta\in\calH$ such that $\|\m^\delta-\m\|_{\calH}\leq \delta$, we have that $\sum_{i=1}^N\|(F^\delta-F^\dagger)_i\|_{\HS}=\mathcal{O}(\delta)$.
    \label{prop_robust_recovery_op}
\end{proposition}
The proof of \Cref{prop_robust_recovery_op} relies on classical arguments in sparse or low-complexity regularization (see for instance \citet[Section 6]{vaiterLowComplexityRegularizations2014} and \cite{burgerConvergenceRatesConvex2004,resmeritaRegularizationIllposedProblems2005,hofmannConvergenceRatesResult2007,scherzerVariationalMethodsImaging2008,grasmairNecessarySufficientConditions2011}). Its proof is postponed to \Cref{appendix_robust_recovery}. As in \Cref{subsec_recovery_nuc_norm_min}, we stress that its main drawback is that it does not provide model consistency guarantees. In particular, it does not allow us to claim that $F^\delta$ is a rank-one operator. Ensuring that this stronger property holds would require showing that a special dual certificate, called the minimal norm dual certificate, is non-degenerate. In the following paragraph, we explain how a simple proxy for this certificate could be obtained.

\paragraph{Pre-certificate.} If $\restriction{\Phi}{T}$ is injective and has closed range, then
\begin{equation}
        \underset{p\in\mathcal{H}}{\mathrm{inf}}~\|p\|_{\mathcal{H}}~~\mathrm{s.t.}~~\mathrm{for~every}~1\leq i\leq N,~[P_T\Phi^*p]_i=u_i\otimes v_i
    \label{def_pre_certif}
\end{equation}
has a unique solution. Indeed, using \citet[Theorem 2.21]{brezisFunctionalAnalysisSobolev2011}, we obtain that $L\eqdef(\restriction{\Phi}{T})^*=P_T\Phi^*$ is surjective. As a result, $LL^*$ is invertible and the pseudo-inverse $L^\dagger$ of $L$ is well-defined and given by $L^\dagger=L^*(LL^*)^{-1}$. This shows that the unique solution $p$ to \eqref{def_pre_certif} is 
\begin{equation*}
    L^\dagger (u_i\otimes v_i)_{1\leq i\leq N}=\restriction{\Phi}{T} (P_T\Phi^*\restriction{\Phi}{T})^{-1}(u_i\otimes v_i)_{1\leq i\leq N}.
\end{equation*} 
In the literature on sparse or low complexity regularization (see for instance \cite{fuchsSparseRepresentationsArbitrary2004} for the case of $\ell^1$ regularization, \cite{Candes2009-fw} for nuclear norm regularization and \cite{vaiterModelSelectionLow2015,vaiterModelConsistencyPartly2018} for the more general case of partly smooth regularizers), the quantity $H=\Phi^*p$ is often called \emph{dual pre-certificate}. It satisfies $P_{T_i}(H_i)=u_i\otimes v_i$ but it is not necessarily a valid dual certificate, in the sense that $\|P_{T_i^\perp}H_i\|$ could be larger than or equal to $1$. Its main interest is that it can be computed by solving a linear system. As a result, it can be used as a simple proxy for the minimal norm dual certificate and proving it is a valid certificate provides model selection guarantees. Although we do not use it in the following, we think investigating its behavior in the settings of \Cref{sec_internal_measurements,sec_calderon} could be an interesting avenue for future works.

\section{Bochner spaces}\label{sec_bochner}

The lifting approach we present in \Cref{sec_internal_measurements,sec_calderon} leads us to consider functions of $2d$ variables with different regularity on the first $d$ and on the last $d$ variables. The natural spaces for this type of functions are  \textit{Bochner spaces}. As explained in this section, these spaces are isometrically isomorphic to spaces of Hilbert-Schmidt operators, which allows the corresponding nuclear norm minimization problems to be addressed as discussed in \Cref{subsec_nuclear_norm_min_op}.
 
For every $m\ge0$, we denote  the Sobolev space of order $m$ on a domain $\Omega\subseteq \mathbb{R}^d$ by $\Hm$, with the usual convention that $H^0(\Omega)=\LD$. Given  a real separable Hilbert space $\calH$, the Hilbert space of $\calH$-valued $H^m$-functions (see \cite{kreuter2015sobolev}) is defined as
\begin{equation*}
    \HmH\eqdef\{F\in L^2(\Omega;\calH): \operatorname{D}^\alpha F\in L^2(\Omega;\calH) \mbox{ for } \lvert\alpha\rvert\leq m\},
\end{equation*}
where $L^2(\Omega;\calH)$ denotes the space of square integrable $\calH$-valued functions (see \citet[Section V.5]{Yosida1980-mb} and \citet[Chapter 2]{kreuter2015sobolev}). Its inner product is given by 
\begin{equation*}
    \langle F,G\rangle_{H^m(\Omega;\calH)}\eqdef\sum\limits_{\lvert\alpha\rvert\leq m} \langle \operatorname{D}^\alpha F,\operatorname{D}^\alpha G\rangle_{L^2(\Omega;\calH)}, \qquad F,G\in H^m(\Omega;\calH).
\end{equation*}
Most classical results about real-valued Sobolev functions generalize to $\mathcal{H}$-valued Sobolev functions. We refer the reader to \cite{arendt2016mapping} for further details. 

We now turn to the following lemma, which states that $\calB_2(\Hm;\calH)$ and $\HmH$ are isomorphic. For the sake of completeness, its proof is given in \Cref{appendix_proof_isomorphism}.
\begin{lemma}\label{lemma isomorphism B2 and Bochner}
    There exists an isometric isomorphism 
    \[
    \mathcal{J}\colon \calB_2(\Hm;\calH)\to \HmH
    \]
    such that, for every $u\in\Hm$ and $v\in\calH$, it holds $\mathcal{J}(u\otimes v)(x)=u(x)v\in\calH$ for almost every $x\in\Omega$.
\end{lemma}
By virtue of this result, in the next sections we will always use the Bochner's notation $\HmH$ and any mention of the rank or nuclear norm of a function shall refer to the rank or nuclear norm of its associated Hilbert–Schmidt operator via this isomorphism. Similarly, with an abuse of notation, we shall denote $\mathcal{J}(u\otimes v)$ simply by $u\otimes v$.

We notice that whenever $\calH$ can be embedded into $L^2(\Omega)$, we can treat the functions in $\HmH$ as functions of $2d$ variables due to the embedding 
\begin{equation*}
    \HmH\hookrightarrow L^2(\Omega;\LD)\cong L^2(\Omega\times\Omega).
\end{equation*}

If $\calH\subseteq L^2(\Omega)$ is a suitable reproducing kernel Hilbert space \cite[Chapter~4]{steinwart-christmann-svm}, it is possible to restrict a function $F\in\HmH\hookrightarrow L^2(\Omega\times\Omega)$ to the diagonal of $\Omega\times\Omega$, see also \cite{BrislawnKernelsTraceClass} and \citet[Section~3]{muller_quantitative_2024}.
\begin{lemma}
   Assume that $\calH\subseteq L^2(\Omega)$ is a reproducing kernel Hilbert space with kernel $K\colon\Omega\times\Omega\to\RR$ satisfying
\begin{equation}\label{equation condition on K for restriction to diag}
    \sup\limits_{x\in\Omega}\|K_x\|_\calH<+\infty.
\end{equation} Then, the map
    \[
     C(\Omega;\calH) \to L^2(\Omega),\qquad F\mapsto  \restriction{F}{\operatorname{diag}(\Omega\times\Omega)} ,
    \]
    where $\restriction{F}{\operatorname{diag}(\Omega\times\Omega)}(x) = F(x,x)$, extends to a unique  bounded linear operator
    \begin{equation}\label{operator restriction to diag}
    \begin{aligned}
        \restriction{\cdot}{\operatorname{diag}(\Omega\times\Omega)}\colon \HmH&\to L^2(\Omega)\\
        F&\mapsto \restriction{F}{\operatorname{diag}(\Omega\times\Omega)}.
    \end{aligned}
\end{equation}
\end{lemma}
\begin{proof}
    It is enough to observe that for $F\in C(\Omega;\calH) \hookrightarrow L^2(\Omega\times\Omega)$ we have
    \begin{equation*}
    \lvert F(x,x)\rvert= \lvert\langle F(x),K_x\rangle_\calH\rvert\leq C \|F(x)\|_\calH, \qquad  x\in\Omega,
\end{equation*}
where $C=\sup_{x\in\Omega}\|K_x\|_\calH$. Then
\[
\|\restriction{F}{\operatorname{diag}(\Omega\times\Omega)}\|_{L^2(\Omega)}^2 \le C^2\int_\Omega \|F(x)\|_\calH^2\,dx =C^2 \|F\|_{L^2(\Omega;\calH)}^2 \le C^2 \|F\|_{\HmH}^2.
\]
Finally, the result follows by the density of $C(\Omega;\calH)$ into $\HmH$.
\end{proof}   
\section{The Calder\'on problem}\label{sec_calderon}
In this section, we investigate the application of the techniques introduced above to the Calder\'on problem.
\subsection{Problem formulation}
Let $\Omega\subset\RR^d$ ($d\geq 2$) be a bounded Lipschitz domain and $\calW$ be a finite-dimensional subspace of $C(\overline\Omega)$. The assumption $\calW\subset C(\overline{\Omega})$ allows us to see $\calW$ as a reproducing kernel Hilbert space when endowed with the $L^2$-norm. In particular, if $\omega_1,\dots,\omega_m$ form an orthonormal basis for $\calW$, then the kernel is given by
\begin{equation}\label{equation expression of K for calW}
    K=\sum\limits_{i=1}^m \omega_i\otimes\omega_i.
\end{equation}
Throughout this section, we denote the $L^2$ norm on $\calW$ by $\|\cdot\|_\calW$.

In order to obtain a well-defined Dirichlet-to-Neumann map 
\begin{equation}\label{equation DN map expression}
\begin{aligned}
    \Lambda_{q^\dagger}\colon H^{1/2}(\partial\Omega)&\to H^{-1/2}(\partial\Omega)\\
    f&\mapsto \restriction{\partial_\nu u^\dagger}{\partial\Omega},
\end{aligned}
\end{equation}
where $u^\dagger\in H^1(\Omega)$ is the unique weak solution to
\begin{equation}\left\{
    \begin{aligned}
        -\Delta u^\dagger +q^\dagger u^\dagger&=0&&\mathrm{in}~\Omega,\\
        u^\dagger&=f&&\mathrm{on}~\partial\Omega,
    \end{aligned}\right.
    \label{eq_calderon}
\end{equation}
we assume in the following that the unknown potential $q^\dagger\in\calW$ is such that $0$ is not an eigenvalue of $-\Delta+q^\dagger$ (see for example \citep[Theorem 2.63]{feldman2019calderon}). This is always the case if $q^\dagger\ge 0$ or whenever $q^\dagger=\Delta\sqrt{\sigma^\dagger}/\sqrt{\sigma^\dagger}$ for some unknown conductivity $\sigma^\dagger\in L^{\infty}_+(\Omega)$ (this corresponds to the classical reduction discussed in \Cref{sub:eit_calderon}). Our aim is to recover $q^\dagger$ from the knowledge of $(\Lambda_{q^\dagger}f_i)_{1\leq i\leq N}$ for finitely many Dirichlet data $(f_i)_{1\leq i\leq N}$. 

Let $\{f_i\}_{i\in\NN}$ be a fixed  orthonormal basis of $H^{1/2}(\partial\Omega)$, with $f_i$ bounded for every $i$. In the following, we use $f_i$ ($i=\toN$) as Dirichlet boundary data and we denote  the unique weak solution of \eqref{eq_calderon} with $f=f_i$ by $u_i^\dagger\in H^1(\Omega)$. 

\subsection{Lifting}
We propose to lift the inverse problem by considering a new unknown. In the current setting, we have multiple PDEs to take into account (one for each boundary datum). We therefore introduce a new unknown for each PDE and lift the problem by considering
\begin{equation*}
    F^\dagger_i\eqdef u_i^\dagger\otimes q^\dagger\in\HUW,\qquad i=\toN.
\end{equation*}
Heuristically, this corresponds to
\[
F_i^\dagger(x,y)=u_i^\dagger(x)q^\dagger(y),\qquad \text{a.e.\ } (x,y)\in \Omega\times\Omega.
\]

By \Cref{equation expression of K for calW}, we notice that $K$ satisfies \Cref{equation condition on K for restriction to diag}. As a result, one can define $\restriction{F}{\diag}\in\LD$ for every ${F\in \HUW}$. With this change of variable, we get $u_i^\dagger=u_{F_i^\dagger}$, where, for every ${F=(F_i)_{1\leq i\leq N}\in\HUW^N}$, we denote by $u_{F_i}\in H^1(\Omega)$ the unique weak solution to
\begin{equation}\left\{\label{eq solved by u_G_i infinite}
    \begin{aligned}
        \Delta u_{F_i} &=\restriction{F_i}{\diag}&&\mathrm{in}~\Omega,\\
        u_{F_i}&=f_i&&\mathrm{on}~\partial\Omega.
    \end{aligned}\right.
\end{equation}

The main interest of the lifting is that the map $F_i\mapsto \restriction{\partial_{\nu} u_{F_i}}{\partial\Omega}$ is \emph{affine}, while $q^\dagger\mapsto \restriction{\partial_{\nu} u_i^\dagger}{\partial\Omega}$ is nonlinear. To obtain a linear mapping from the unknown to the observations, we define, for every $F_i\in\HUW$, the function $v_{F_i}\eqdef u_{F_i}-\tilde{f_i}\in H^1_0(\Omega)$, where $\tilde{f_i}\in H^1(\Omega)$ is the harmonic extension of $f_i\in H^{1/2}(\partial\Omega)$. In particular, since $u_{F^\dagger_i}=u^\dagger_i$, it holds that $\restriction{\partial_\nu v_{F^\dagger_i}}{\partial\Omega}=\restriction{\partial_\nu u^\dagger_{i}}{\partial\Omega}-\restriction{\partial_\nu \tilde{f_i}}{\partial\Omega}$.

Now we would like to find a bounded linear operator $\Phi$ from the set of lifted variables $H^1(\Omega;\calW)^N$ to a separable Hilbert space $\calH$ and a vector $z\in\calH$ depending only on the available measurements $(\restriction{\partial_\nu u^\dagger_{i}}{\partial\Omega})_{1\leq i\leq N}$ such that $F^\dagger=(F_i^\dagger)_{1\leq i\leq N}$ is the unique solution to  
\begin{equation}
    \mathrm{find}~F\in\HUW^N~\mathrm{s.t.}~\Phi F=\m~\mathrm{and}~\mathrm{rank}(F_i)=1~\mbox{ for every }~i=\toN.
        \tag{$\mathcal{P}_{\mathrm{lifted}}$}
        \label{pb_lifted_calderon_infinite_measurements}
\end{equation}
We construct such an operator $\Phi:F\mapsto (\Phi_1 F,\Phi_2 F,\Phi_3 F)$ and a vector $z=(z_1,z_2,z_3)$ in three steps below. The first component $\Phi_1$ enforces the constraint that $F$ yields the same boundary measurements as $F^\dagger$. The second component enforces the constraint that $F_i(\cdot,y)$ must be proportional to $u_{F_i}$ for every $1\leq i\leq N$. Finally, the third component enforces the constraint that the functions $F_i$ have the same component along their last $d$ variables.

\paragraph{Consistency with the measurements.} In order to take into account the Neumann data, we define
\begin{equation*}
    \begin{aligned}
        \Phi_1 \colon \HUW^N &\to H^{-1/2}(\partial\Omega)^N \\
        F &\mapsto (\restriction{\partial_\nu v_{F_i}}{\partial\Omega})_{1\leq i\leq N},
    \end{aligned}
\end{equation*}
together with the vector $\m_1\eqdef \Phi_1F^\dagger=(\restriction{\partial_\nu v_{F^\dagger_i}}{\partial\Omega})_{1\leq i\leq N}$.

\paragraph{Proportionality to the solution of the PDE.} Now, we leverage the information that $F^\dagger_i(\cdot,y)=q^\dagger(y)u_i^\dagger$ is proportional to $u_i^\dagger$ for every $i=\toN$. In this regard, assuming we know $\int_\Omega q^\dagger\neq 0$, we impose, for every $i=\toN$, that 
\begin{equation*}
    \int_\Omega F_i(\cdot,y)dy = \bigg[\int_\Omega q^\dagger\bigg] u_{F_i}.
    \end{equation*}
This is equivalent to imposing that $\Phi_2 F=z_2$, where $\Phi_2$ is the bounded linear operator defined by
\begin{equation*}
    \begin{aligned}
        \Phi_2 \colon \HUW^N &\to H^{1}(\Omega)^N \\
        F &\mapsto \bigg(\int_\Omega F_i(\cdot,y)dy - \bigg[\int_\Omega q^\dagger\bigg] v_{F_i}\bigg)_{1\leq i\leq N}
    \end{aligned}
\end{equation*} 
and $\m_2\eqdef \Phi_2 F^\dagger=([\int_\Omega q^\dagger]\Tilde{f_i})_{1\leq i\leq N}$.
\begin{remark}\label{remark int on K}
It is worth observing that all the analysis presented in this section would work also if we replaced the quantity $\int_\Omega q^\dagger$ by $g(q^\dagger)$ for any continuous linear functional $g$. For instance, two alternatives could be $g(q^\dagger)=\int_K q^\dagger$ for some small subdomain $K\subset\Omega$ or $g(q^\dagger)=q^\dagger(y_0)$ for some $y_0\in\overline{\Omega}$. The choice of the integral on $\Omega$ has been made for uniformity with the setting of \Cref{sec_internal_measurements}, where this leads to simpler calculations. 
\end{remark}
\paragraph{Equality along the last $d$ variables.} The third constraint we impose leverages the fact that the unknowns $F_i^\dagger$ have the same component $q^\dagger$ along their last $d$ variables. To achieve this, we rely on the existence of a trace in $L^2(\partial\Omega;\calW)$ for every function in~$\HUW$ (see \citet[Theorem 7.11]{arendt2016mapping}) and, for every $1\leq i\leq N-1$ and $j>i$, we impose the following identity in $\calW$:
\begin{equation*}
    f_j(x)F_i(x,\cdot)=f_i(x)F_j(x,\cdot),\qquad \text{a.e.\ } x\in\partial\Omega.
\end{equation*}
This leads us to define the bounded linear operator
\begin{equation*}
    \begin{aligned}
        \Phi_3 \colon \HUW^N &\to L^2(\partial\Omega;\calW)^{\binom{N}{2}} \\
        F &\mapsto (\Phi_3^{i,j}(F))_{1\leq i\leq N-1,j>i},
    \end{aligned}
\end{equation*}
where $(\Phi_3^{i,j}F)(x)= f_j(x)F_i(x)-f_i(x)F_j(x)$.
Since $u_i^\dagger=f_i$ on $\partial\Omega$, we see that
\begin{equation*}
    (\Phi_3^{i,j}F^\dagger)(x)=f_j(x)F^\dagger_i(x)-f_i(x)F^\dagger_j(x)=f_j(x)f_i(x)q^\dagger-f_i(x)f_j(x)q^\dagger=0,
\end{equation*}
so that $\m_3\eqdef \Phi_3 F^\dagger=0$.

Now, combining the three operators defined above, we obtain the bounded linear operator
\begin{equation*}
        \Phi \colon \HUW^N \to \calH,\qquad 
        \Phi(F) = (
            \Phi_1 F,
            \Phi_2 F,
            \Phi_3 F),
\end{equation*}
where $\calH\eqdef \calH_1\times\calH_2\times\calH_3$ with $\calH_1\eqdef H^{-1/2}(\partial\Omega)^N$, $\calH_2\eqdef H^1(\Omega)^N$ and $\calH_3\eqdef L^2(\partial\Omega;\mathcal{W})^{\binom{N}{2}}$. The measurements associated to the unknown $F^\dagger$ are
\begin{equation*}
    \m\eqdef \Phi F^\dagger=\bigg[\bigg(\partial_\nu \restriction{v_{F^\dagger_i}}{\partial\Omega}\bigg)_{1\leq i\leq N},\bigg(\bigg[\int_\Omega q^\dagger\bigg]\Tilde{f_i}\bigg)_{1\leq i\leq N}, 0\bigg].
\end{equation*}

Our first result aims at showing that $F^\dagger$ is the only solution to \Cref{pb_lifted_calderon_infinite_measurements}. As shown in the proof of the next proposition, this result heavily depends on identifiability results for the Calder\'on problem.  
To prove the result in our setting, we would need an identifiability result in the case of a finite number of samples from the Cauchy data, that is to say pairs of Dirichlet and Neumann boundary values (see for instance \citet[Definition 2.68]{feldman2019calderon}). We are not aware of such a result in the literature, and the available results for the Dirichlet to Neumann map \cite{harrach2019uniqueness,alberti2022inverse,alberti2022infinite} cannot be easily extended to this setting. Still, the result can be proved in the case where infinitely many measurements are available, namely when we observe the Neumann data associated to the whole basis $\{f_i\}_{i\in\NN}$. This is the object of the proposition below, which can be seen as a first consistency check for the lifting we propose.

\begin{proposition}\label{proposition equivalence original prob and rank min prob infinite}
Let $(F_i)_{i\in\NN}$ be such that $F_i\in\HUW$, $\operatorname{rank}(F_i)=1$ and 
\begin{empheq}[left = \empheqlbrace]{align*}
        &\restriction{\partial_\nu v_{F_i}}{\partial\Omega} = \restriction{\partial_\nu v_{F^\dagger_i}}{\partial\Omega}, \\
        &\int_\Omega F_i(\cdot,y)dy=\bigg[\int_\Omega q^\dagger\bigg]u_{F_i}, \\
        &f_j(x)F_i(x)=
        f_i(x)F_j(x), \qquad \mathrm{for~a.e.\ }x\in\partial\Omega,
\end{empheq} 
for every $i\in\NN$ and $j>i$. Then $F_i=u_i^\dagger\otimes q^\dagger$ for every $i\in\NN$.
\end{proposition}
\begin{proof}
   For every $i\in\NN$ we set $F_i=a_i\otimes b_i$ with $a_i\in H^1(\Omega)$ and $b_i\in\mathcal{W}$. Now, the set of constraints gives us that
    \begin{empheq}[left = \empheqlbrace]{align}\label{eq: phi1=b1 infinite}
        &\restriction{\partial_\nu u_{F_i}}{\partial\Omega} = \restriction{\partial_\nu u_i^\dagger}{\partial\Omega}, \\
        \label{eq: phi2=b2 infinite}
        &\bigg[\int_\Omega b_i\bigg]a_i=\bigg[\int_\Omega q^\dagger\bigg]u_{F_i}, \\
        \label{eq: phi3=b3 infinite}
        &f_j(x)a_i(x)b_i(y)=f_i(x)a_j(x)b_j(y), \qquad \mathrm{for~a.e.~}(x,y)\in\partial\Omega\times\Omega.
    \end{empheq} 
     We stress that $\int_\Omega b_i\neq 0$ for every $i\in\NN$. Indeed, if $\int_\Omega b_i=0$, since $\int_\Omega q^\dagger\neq0$, we would obtain $u_{F_i}=0$, which contradicts $\restriction{u_{F_i}}{\partial\Omega}=f_i\neq0$.
    Hence, using \Cref{eq: phi2=b2 infinite} and \Cref{eq: phi3=b3 infinite}, we obtain, for every $i\in\NN$ and $j>i$, that
    \begin{equation}\label{expression for G^2 infinite}
        \frac{1}{\int_\Omega b_i} b_i = \frac{1}{\int_\Omega b_j} b_j \eqdef b.
    \end{equation}
    By \Cref{eq solved by u_G_i infinite}, \Cref{eq: phi2=b2 infinite} and \Cref{expression for G^2 infinite}, we have that
    \begin{equation}\label{PDE for u_G_i in prop rank min infinite}
        -\Delta u_{F_i}+\bigg[\int_\Omega q^\dagger\bigg] bu_{F_i}=0\qquad \text{in $\Omega$}
    \end{equation}
    for every $i\in\NN$. Since $\restriction{u_{F_i}}{\partial\Omega}=f_i$ for every $i\in\NN$ and $\{f_i\}_{i\in\NN}$ is an orthonormal basis of $H^{1/2}(\partial\Omega)$, we can use the Fredholm alternative (see for example \citet[Section 6.2.3]{evans2022partial}) to conclude that $0$ is not an eigenvalue of $-\Delta+[\int_\Omega q^\dagger] b$ with Dirichlet boundary conditions. Indeed, if $\{w_1,\dots,w_n\}$ is a basis of $\operatorname{Ker}(-\Delta+[\int_\Omega q^\dagger] b)$, the Fredholm alternative ensures that $\langle \partial_\nu w_j,f_i\rangle_{\partial\Omega}=0$ for every $i\in\NN$ and $j=1,\dots,n$. This shows that $\restriction{\partial_\nu w_j}{\partial\Omega}=0$ for every $j=1,\dots,n$. Using the unique continuation principle \citep{tataru1999carleman}, we conclude that $w_j=0$ for every $j=1,\dots,n$.
    
    As a result, in view of \eqref{eq: phi1=b1 infinite}, relying on the identifiability of $L^\infty(\Omega)$ potentials with infinitely many measurements (see for instance \citet{sylvester1987global} for $d\ge3$ and \citet{Bukhgeim+2008+19+33} for $d=2$), we obtain that 
    $$
    q^\dagger=\bigg[\int_\Omega q^\dagger\bigg] b= \frac{\int_\Omega q^\dagger}{\int_\Omega b_i}b_i
    $$
    for every $i\in \NN$. This in turn yields
    \begin{equation}\label{expression for G_i^2 infinite}
    b_i=\frac{\int_\Omega b_i}{\int_\Omega q^\dagger}q^\dagger
    \end{equation}
    for every $i\in\NN$. As a result, using \Cref{expression for G^2 infinite} and \Cref{PDE for u_G_i in prop rank min infinite}, we obtain that $u_{F_i}=u_i^\dagger$ for every $i\in\NN$, which in turn yields
\begin{equation}\label{expression for G_i^1 infinite}
a_i = \frac{\int_\Omega q^\dagger}{\int_\Omega b_i}u_i^\dagger.
\end{equation}
Finally, using \Cref{expression for G_i^2 infinite} and \Cref{expression for G_i^1 infinite}, we obtain
\(
F_i = a_i\otimes b_i = u_i^\dagger \otimes q^\dagger .
\)
\end{proof}

\subsection{Convex relaxation}\label{sec_convex_relaxation_calderon}
Let $\sigma_i=\|u_i^\dagger\|_{H^1(\Omega)}\|q^\dagger\|_\calW>0$ ($1\leq i\leq N$). From now on, we denote the normalized versions of $q^\dagger$ and $(u_i^\dagger)_{1\leq i\leq N}$ by $q$ and $(u_i)_{1\leq i\leq N}$, namely $q = q^\dagger/\|q^\dagger\|_\calW$ and $u_i = u_i^\dagger/\|u_i^\dagger\|_{H^1(\Omega)}$, and write $F^\dagger_i=\sigma_i u_i\otimes q$.

In order to obtain a convex minimization problem from \Cref{pb_lifted_calderon_infinite_measurements}, we use the nuclear norm as a convex proxy for the rank. As in Section~\ref{subsec_nuclear_norm_min_op}, we propose to relax the rank-one constraint on the $(F_i)_{1\leq i\leq N}$ by minimizing the sum of their nuclear norms, which is a proper convex lower semi-continuous function. This leads to the following convex relaxation of~\Cref{pb_lifted_calderon_infinite_measurements}:
\begin{equation}
    \underset{F\in \HUW^N}{\min}~\sum\limits_{i=1}^N \|F_i\|_*~~\mathrm{s.t.}~~\Phi F=\m.
    \tag{$\mathcal{P}_{\mathrm{relaxed}}$}
    \label{primal_calderon}
\end{equation}

\subsubsection{Exact and robust recovery}
\label{subsec_recovery_calderon}
Relying on the results of \Cref{subsec_nuclear_norm_min_op}, we can show that a non-degenerate source condition implies an exact recovery result for the unknown $F^\dagger$.
\begin{proposition}\label{proposition exact recovery calderon infinite}
    If there exists $N\in\NN$ such that the non-degenerate source condition
        \begin{equation}
        \begin{aligned}
           &\mbox{there exists}~H\in\operatorname{Im}(\Phi^*)~\mbox{s.t.}~\mathrm{for~every}~1\leq i\leq N,~
           H_i=u_i\otimes q+W_i\\&\mathrm{with}~W_iu_i=0,~W_i^*q=0~\mathrm{and}~\|W_i\|<1
           \end{aligned}
        \tag{NDSC}
        \label{ndsc_calderon}
    \end{equation}
    holds, then $F^\dagger=(\sigma_i u_i\otimes q)_{1\leq i\leq N}=(u^\dagger_i\otimes q^\dagger)_{1\leq i\leq N}$ is the unique solution to \eqref{primal_calderon}. Moreover, $q^\dagger$ is uniquely determined by $F_i^\dagger$ via the identity
    \begin{equation*}
        q^\dagger=\frac{1}{\int_{\partial\Omega} \lvert f_i\rvert^2}\int_{\partial\Omega}\restriction{F_i^\dagger}{\partial\Omega}(x,\cdot)f_i(x)dx.
    \end{equation*}
\end{proposition}
\begin{proof}
To prove the first part of the statement, we apply \Cref{prop_exact_recovery_op} by noticing that $\Phi$ is injective on $\mathcal C=\{\alpha_iu_i\otimes q:\alpha_i\ge0\}$. Indeed, $\Phi_1(\alpha_i u_i\otimes q)_{i=1}^N=0$ gives that $\restriction{\partial_\nu v_{\alpha_i u_i\otimes q}}{\partial\Omega}=0$ for every $1\le i\le N$. By the unique continuation principle, this implies $v_{\alpha_i u_i\otimes q}=0$ in $\Omega$ for every $1\le i\le N$. Now, since $\int_\Omega q\neq0$, $\Phi_2(\alpha_i u_i\otimes q)_{i=1}^N=0$ implies that $\alpha_iu_i=v_{\alpha_i u_i\otimes q}=0$. Therefore, $\alpha_i u_i\otimes q=0$ for every $1\le i\le N$. The second part is an immediate consequence of the fact that $\restriction{F_i^\dagger}{\partial\Omega}(x,\cdot)=f_i(x)q^\dagger$. 
\end{proof}
\begin{remark}
    The aim of our method is to find a convex problem whose solution allows for the reconstruction of an unknown potential for the Calder\'on problem. Therefore, we expect to need a sufficiently large number of measurements $N$ to prove the existence of a dual certificate. In other words, we expect that the validity of condition \Cref{ndsc_calderon} depends on $N$ and that taking a larger $N$ should make its verification easier.
\end{remark}
We now turn to the question of robust recovery. Even though the Calder\'on problem is severely ill-posed \cite{alessandrini1988stable,mandache2001exponential}, it is well known that when reducing to a finite-dimensional unknown $q^\dagger$, the inverse problem becomes Lipschitz stable (see e.g.\ \cite{alessandrini2005lipschitz,bacchelli2006lipschitz, beretta2013lipschitz,harrach2019uniqueness,alberti2022infinite}). So it is expected that a linear rate for the recovery  may be obtained in this context. A regularization approach with a TV penalty term was studied in \cite{rondi-2016,felisi-rondi-2024}, with a precise control of the reconstruction error with respect to the noise level.

We assume that the boundary measurements are only known up to an additive noise of amplitude $\delta>0$. To be more precise, we assume to know $\tilde{z}_1^\delta$ such that the boundary measurements $\tilde{z}_1\eqdef(\restriction{\partial_\nu u^\dagger_i}{\partial\Omega})_{1\leq i\leq N}$ satisfy $\|\tilde{z}_1^\delta-\tilde{z}_1\|_{\calH_1}\leq \delta$. Recalling that $$z_1=(\restriction{\partial_\nu v_{F^\dagger_i}}{\partial\Omega})_{1\leq i\leq N}=(\restriction{\partial_\nu u^\dagger_{i}}{\partial\Omega}-\restriction{\partial_\nu \tilde{f}_i}{\partial\Omega})_{1\leq i\leq N},$$ defining $z_1^\delta\eqdef \tilde{z}_1^\delta -(\restriction{\partial_{\nu}\tilde{f}_i}{\partial\Omega})_{1\leq i\leq N}$, we have $\|z_1^\delta -z_1\|_{\calH_1}\leq\delta$. We assume that $z_2$ and $z_3$ are observed exactly, because they do not corresponds to measurements, but rather to the modeling.

We wish to estimate the unknown $F^\dagger$ by solving
\begin{equation}
        \underset{F\in\HUW^N}{\mathrm{min}}~\frac{1}{2}\|\Phi_1 F-\m_1^{\delta}\|_{\calH_1}^2+\lambda\sum\limits_{i=1}^N\|F_i\|_*~~\mathrm{s.t.}~~\Phi_2F=z_2,~\Phi_3F=0
\tag{$\mathcal{P}^{\lambda,\delta}_{\mathrm{relaxed}}$}
\label{pb_relaxed_noisy infinite}
\end{equation}
for some parameter $\lambda>0$. We assume in the following that the basis $\{f_i\}_{i\in\NN}$ is chosen such that $f_1$ has a positive lower bound, i.e.\ $\lvert f_1(x)\rvert\geq C>0$, a.e.\ $x\in\partial\Omega$ for some positive constant $C>0$. This allows us to derive the following stable recovery result.
\begin{proposition}
   Let $\delta$ and $c$ be positive constants. If there exists $N\in\NN$ such that the non-degenerate source condition \cref{ndsc_calderon} holds then, for any minimizer $F^\delta$ of \cref{pb_relaxed_noisy infinite} with $\lambda =c\delta$ and $\m_1^{\delta}\in\mathcal{H}_1$ such that $\|z_1^\delta -z_1\|_{\calH_1}\leq\delta$ holds, we have that $\|F^{\delta}-F^\dagger\|_{\HUW^N}=\mathcal{O}(\delta)$. Moreover, setting 
   \begin{equation*}
        q^\delta=\frac{1}{\int_{\partial\Omega} \lvert f_i\rvert^2}\int_{\partial\Omega}\restriction{F_i^\delta}{\partial\Omega}(x,\cdot)f_i(x)dx
   \end{equation*}
  for some $i=1,\dots,N$, it holds that $\|q^\delta-q^\dagger\|_{\calW}=\mathcal{O}(\delta)$.
   \label{robust_recovery infinite}
\end{proposition}
\begin{proof}
By \Cref{lemma closed range infinite} below,  $\restriction{\Phi}{T}$ is injective and has closed range. Thus, the first part of the statement follows by using \Cref{prop_robust_recovery_op}.

    Let us now prove the second part of the statement. We recall that, by \Cref{proposition exact recovery calderon infinite}, we have $$q^\dagger=\frac{1}{\int_{\partial\Omega} \lvert f_i\rvert^2}\int_{\partial\Omega}\restriction{F_i^\dagger}{\partial\Omega}(x,\cdot)f_i(x)dx.$$ Therefore, by using the Cauchy-Schwartz inequality and the Fubini theorem, we conclude that
    \begin{equation*}
    \begin{aligned}
        \|q^\delta-q^\dagger\|^2_{\calW}
        &\le\frac{1}{\big(\int_{\partial\Omega} \lvert f_i\rvert^2\big)^2}\int_\Omega\bigg(\int_{\partial\Omega}\lvert\restriction{F_i^\delta}{\partial\Omega}(x,y)-\restriction{F_i^\dagger}{\partial\Omega}(x,y)\rvert\lvert f_i(x)\rvert dx\bigg)^2dy\\
        &\le\frac{1}{\int_{\partial\Omega} \lvert f_i\rvert^2}\int_\Omega\int_{\partial\Omega}\lvert\restriction{F_i^\delta}{\partial\Omega}(x,y)-\restriction{F_i^\dagger}{\partial\Omega}(x,y)\rvert^2dxdy\\
        &=\frac{1}{\int_{\partial\Omega} \lvert f_i\rvert^2}\int_{\partial\Omega}\int_\Omega\lvert\restriction{F_i^\delta}{\partial\Omega}(x,y)-\restriction{F_i^\dagger}{\partial\Omega}(x,y)\rvert^2dydx\\
        &=\|\restriction{F_i^\delta}{\partial\Omega}-\restriction{F_i^\dagger}{\partial\Omega}\|^2_{L^2(\partial\Omega;\calW)}.
    \end{aligned}
    \end{equation*}
    Finally, the boundedness of the trace operator for vector-valued functions \citet[Theorem 7.11]{arendt2016mapping} gives the result.
\end{proof}
We now turn to the proof that $\restriction{\Phi}{T}$ is injective and has closed range.

\begin{lemma}\label{lemma closed range infinite}
    There exist $N\in\NN$ and $C_\Phi>0$ such that, for every $F\in T$,
    \begin{equation*}
        \sum\limits_{i=1}^N\|F_i\|_{\HUW}\le C_\Phi \|\Phi F\|_{\calH}.
    \end{equation*}
\end{lemma}
\begin{proof}
    Let $F_i = u_i \otimes a_i + b_i \otimes q\in T_i$ ($i=\toN$), see \eqref{eq:TandTi}. The expression of $\Phi_2$ yields
    \begin{equation}\label{equation for b_i+alpha_iu_i closed range T infinite}
        b_i=-\alpha_i u_i +\frac{1}{\int_\Omega q}\Phi_2^iF+\|q^\dagger\|_\calW\,v_{F_i},\qquad i=\toN,
    \end{equation}
    with $\alpha_i\eqdef[\int_\Omega a_i]/[\int_\Omega q]$. For every $i=\toN$, we set $r_i\eqdef\frac{a_i-\alpha_iq}{\|u_i^\dagger\|_{\HU}}$.
    Since $\Delta v_{F_i}=\restriction{F_i}{\diag}$ in $\Omega$, $v_{F_i}=0$ on $\partial\Omega$, by \Cref{equation for b_i+alpha_iu_i closed range T infinite}, we obtain
\begin{equation}\label{eq PDE for v_G_i closed range infinite}
    \left\{\begin{aligned}
        -\Delta v_{F_i}+q^\dagger v_{F_i}&=-r_iu_i^\dagger-\frac{\Phi_2^iF}{\int_\Omega q}q&&~\mathrm{in}~\Omega,\\
        \restriction{v_{F_i}}{\partial\Omega}&=0&&~\mathrm{on}~\partial\Omega.
    \end{aligned}\right.
\end{equation}
As a result, for every $i=\toN$, we get 
\begin{equation}\label{eq bound for b_i+alpha_iu_i closed range infinite}
\begin{aligned}   \|b_i+\alpha_iu_i\|_{\HU}
&\lesssim\|\Phi_2^iF\|_{\HU}+\|v_{F_i}\|_{\HU}\qquad&&\text{by \Cref{equation for b_i+alpha_iu_i closed range T infinite}}\\
&\lesssim\|\Phi_2^iF\|_{\HU}+\|r_i\|_\calW\|u_i^\dagger\|_{\HU}\qquad&&\text{by \Cref{eq PDE for v_G_i closed range infinite}},
\end{aligned}
\end{equation}
where the implicit constant depends on $\Omega$, $q^\dagger$ and the dimension of $\calW$.
Now, for almost every $x\in\partial\Omega$, the expression of $\Phi_3$ and \Cref{equation for b_i+alpha_iu_i closed range T infinite} yield 
\begin{equation}\label{equation Phi3 closed range infinite}
\begin{aligned}
    f_i(x)f_j(x)(r_i-r_j)=\Phi_3^{i,j}F(x)+f_i(x)\frac{1}{\int_\Omega q}\Phi_2^jF(x)q-f_j(x)\frac{1}{\int_\Omega q}\Phi_2^iF(x)q.
\end{aligned}
\end{equation}
By using that $f_1$ has a positive lower bound and that $\int_{\partial\Omega}\lvert f_j\rvert^2\geq \min_{l=\toN}\int_{\partial\Omega}\lvert f_l\rvert^2>0$, we obtain
\begin{equation}\label{eq_bound_ai-aj}
    \|r_1-r_i\|_\calW\lesssim \|\Phi_3^{i,1}F\|_{L^2(\partial\Omega;\calW)}+\|\Phi_2^1F\|_{\HU}+\|\Phi_2^iF\|_{\HU},
\end{equation}
where the constant depends also on $\Omega$, $q^\dagger$, the lower bound of $f_1$ and the $L^2$-norms of the $f_i$s.
Let us define 
\begin{equation}\label{expression of p_i}
    p_i\eqdef (r_1-r_i)u_i^\dagger-\frac{\Phi_2^iF}{\int_\Omega q}q,\qquad i=\toN
\end{equation}
and $A\colon g\in\LD\mapsto\restriction{\partial_\nu w}{\partial\Omega}\in H^{-1/2}(\partial\Omega)$, where $w\in H_0^1(\Omega)$ is the unique weak solution to
\begin{equation*}
    \left\{\begin{aligned}
        -\Delta w+q^\dagger w&=g&&~\mathrm{in}~\Omega,\\
        \restriction{w}{\partial\Omega}&=0&&~\mathrm{on}~\partial\Omega.
    \end{aligned}\right.
\end{equation*}
By \Cref{eq PDE for v_G_i closed range infinite} we obtain
\begin{equation}\label{eq modified PDE for v_Gi calderon}
    \left\{\begin{aligned}
        -\Delta v_{F_i}+q^\dagger v_{F_i}&=p_i-r_1u_i^\dagger&&~\mathrm{in}~\Omega,\\
        \restriction{v_{F_i}}{\partial\Omega}&=0&&~\mathrm{on}~\partial\Omega.
    \end{aligned}\right.
\end{equation}
Let us denote  the forward map for the Calder\'on problem by $\Psi$ (see \Cref{appendix_frechet_derivative}). From \Cref{eq modified PDE for v_Gi calderon} and \Cref{lemma expression and injectivity F'}, we get
\begin{equation}\label{eq for F' in closed range infinite}
    \Phi_1^iF-Ap_i=\Psi'[q^\dagger](r_1)f_i,
\end{equation}
where we recall that $\Phi_1^iF=\restriction{\partial_\nu v_{F_i}}{\partial\Omega}$.

Since $\calW$ is finite-dimensional and $\Psi'[q^\dagger]$ is injective (see \Cref{lemma expression and injectivity F'}), we have that there exists $M>0$ such that $\|w\|_\calW\leq M\|\Psi'[q^\dagger](w)\|$ for every $w\in\calW$.
Let us denote  the projection onto $\operatorname{span}\{f_1,\dots,f_N\}$ by $P_N$ and  the identity operator on $H^{1/2}(\partial\Omega)$ by $I$. 
The argument we are presenting is taken from the proof of Theorem 2 (i) in \cite{alberti2022infinite}. By the continuity of the mapping 
\begin{equation*}
    \calW\ni w\mapsto \|\Psi'[q^\dagger](w)(I-P_N)\|,
\end{equation*}
and by the compactness of the unit sphere in $\calW$, we get that for every $N\in\NN$ there exists $w_N\in\calW$, $\|w_N\|_\calW=1$ such that
\begin{equation*}
    s_N\eqdef\sup_{\|w\|_\calW=1}\|\Psi'[q^\dagger](w)(I-P_N)\|=\|\Psi'[q^\dagger](w_N)(I-P_N)\|.
\end{equation*}
We now want to show that $s_N\to0$ as $N\to+\infty$. Let $(s_{N_j})_{j\in\NN}$ be a subsequence. By a general topology argument, it suffices to show that there exists a subsequence of $(s_{N_j})_{j\in\NN}$ converging to 0. By the compactness of the sphere in $\calW$, the sequence $(w_{N_j})_{j\in\NN}$ admits a subsequence (that we denote by $w_{N_j}$ to keep the notation simpler) such that $w_{N_j}\to w_*$ as $j\to+\infty$ for some $w_*\in\calW$, $\|w_*\|_\calW=1$. Thus, we obtain
\begin{align*}
    s_{N_j}
    &=\|\Psi'[q^\dagger](w_{N_j})(I-P_{N_j})\|\\
    &\leq\|\Psi'[q^\dagger](w_{N_j}-w_*)(I-P_{N_j})\|+\|\Psi'[q^\dagger](w_*)(I-P_{N_j})\|\\
    &\leq 2\|\Psi'[q^\dagger]\|\|w_{N_j}-w_*\|_\calW+\|\Psi'[q^\dagger](w_*)(I-P_{N_j})\|,
\end{align*}
where we used $\|P_{N_j}\|=1$.
By the compactness of $\Psi'[q^\dagger](w_*)$ showed in \Cref{lemma compactness F'}, we have that the second term goes to 0. Therefore, since $w_{N_j}\to w_*$, we conclude that $s_{N_j}\to0$ as $j\to+\infty$. Thus, we have $s_N\to0$ as $N\to+\infty$. Therefore, by taking $w=r_1$, we have that there exists $N\in\NN$ such that
\begin{equation*}
    \|\Psi'[q^\dagger](r_1)(I-P_N)\|\leq \frac{1}{2M}\|r_1\|_\calW.
\end{equation*}
Therefore, we get
\begin{align*}
    \|r_1\|_\calW
    &\leq M\|\Psi'[q^\dagger](r_1)(I-P_N)\|+M\|\Psi'[q^\dagger](r_1)P_N\|\\
    &\leq\frac{1}{2}\|r_1\|_\calW+M\|\Psi'[q^\dagger](r_1)P_N\|,
\end{align*}
from which we obtain
\begin{equation*}
    \|r_1\|_\calW\leq2M\|\Psi'[q^\dagger](r_1)P_N\|.
\end{equation*} 
As a result, we get
\begin{equation*}
\begin{aligned}
    \|r_1\|_\calW
    &\lesssim\|\Psi'[q^\dagger](r_1)P_N\|\\
    &\lesssim\|\Phi_1F\|_{H^{-1/2}(\partial\Omega)^N}+\|(p_i)_{1\leq i\leq N}\|_{\LD^N}\qquad\mathrm{by~\Cref{eq for F' in closed range infinite}},
\end{aligned}
\end{equation*}
with the implicit constant depending on $M$, $q^\dagger$ and the $f_i$s.

Now, by the expression of $p_i$ \eqref{expression of p_i}, we have
\begin{equation*}
    \|p_i\|_{\LD}\lesssim\|r_1-r_i\|_\calW+\|\Phi_2^iF\|_{\HU}, \qquad i=\toN.
\end{equation*}
Therefore, using \Cref{eq_bound_ai-aj}, we get
\[
\|r_i\|_\calW \le \|r_i-r_1\|_\calW + \|r_1\|_\calW \lesssim \|\Phi_1F\|_{H^{-1/2}(\partial\Omega)^N}+\|\Phi_2F\|_{\HU^N}+\|\Phi_3F\|_{L^2(\partial\Omega,\calW)^{\binom{N}{2}}}.
\]
Finally, using \Cref{eq bound for b_i+alpha_iu_i closed range infinite}, recalling that $\|q\|_{\calW}=1$ and observing that
\begin{equation*}
    \|u_i\otimes a_i+b_i\otimes q\|_{\HUW}\leq \|u_i^\dagger\|_{\HU}\|r_i\|_\calW+\|b_i+\alpha_iu_i\|_{\HU},
\end{equation*}
we conclude the proof.
\end{proof}
\subsubsection{Dual certificate}\label{subsubsec_dual_certificate_infinite}
Considering the results of \Cref{subsec_recovery_calderon}, the natural question that remains unanswered is the validity of the non-degenerate source condition \Cref{ndsc_calderon}. We have currently been unable to verify that \eqref{ndsc_calderon} holds under suitable assumptions on the unknown $q^\dagger$. The main difficulty we faced is the complicated form of the forward operator $\Phi$. This makes the construction of dual certificates, which are element of $\mathrm{Im}(\Phi^*)$, highly challenging.

As explained in \Cref{subsec_nuclear_norm_min_op}, our proof that $\restriction{\Phi}{T}$ is injective and has closed range allows us to define a \emph{dual pre-certificate}, which is a simple proxy for the minimal norm dual certificate. Indeed, defining $p\eqdef (P_T\Phi^*)^\dagger((u_i\otimes q)_{1\le i\le N})$, we have that the pre-certificate $H\eqdef\Phi^*p$ satisfies $P_{T_i}(H_i)=u_i\otimes q$ for every ${i=\toN}$. The difficulty is then to find conditions under which $\|P_{T_i^\perp}(H_i)\|\leq 1$ for every $i=\toN$, which would imply the validity of the non-degenerate source condition. 

In the next section, we consider a simpler inverse problem for which the condition \Cref{ndsc_calderon} can be proved.
\section{Internal measurements}\label{sec_internal_measurements}
In this section we consider a toy inverse problem where internal measurements (rather than boundary measurements, as in the Calder\'on case) are available. This is the case for hybrid or coupled-physics inverse problems \citep{BAL-2012,kuchment-2012,2017-ammari-etal,alberti-capdeboscq-2018}. In particular, the model considered here is common in quantitative photoacoustic tomography in a diffusive regime \citep{BAL-REN-2011}. 

We apply the lifting approach to this setting and derive a suitable non-degenerate source condition that guarantees exact and stable recovery. We conclude by verifying this condition under specific structural assumptions on the unknown. This example serves to demonstrate the consistency of our approach in a simplified setting.

\subsection{Problem formulation}
Let $\Omega$ be a connected bounded open subset of $\RR^d$ ($d\in\{1,2,3\}$) with $C^2$ boundary and $q^\dagger\in L^{\infty}(\Omega)$ be a positive potential. Consider a fixed positive Dirichlet datum $f\in H^{3/2}(\partial\Omega)$ and let $u^\dagger\in H^2(\Omega)$ be the unique strong solution to
\begin{equation}\left\{
    \begin{aligned}
        -\Delta u^\dagger+q^\dagger u^\dagger&=0&&\mathrm{in}~\Omega,\\
        u^\dagger&=f&&\mathrm{on}~\partial\Omega.
    \end{aligned}\right.
    \label{ip_internal_meas}
\end{equation}
The inverse problem we wish to solve consists of recovering $q^\dagger$ from the knowledge of $u^\dagger$ in $\Omega$. 

In this section, we prove that if $q^\dagger$ does not vary too much (that is to say that the difference $\sup q^\dagger -\inf q^\dagger$ is small enough), it can be recovered as the unique solution of a nuclear norm optimization problem. We stress that this result is only useful as a consistency check for our lifting approach, as the considered inverse problem can be directly solved by using $q^\dagger=\Delta u^\dagger/u^\dagger$, since, by the strong maximum principle \citep[Theorem 9.6]{Gilbarg2001-ry},  $u^\dagger$ is positive on $\overline{\Omega}$.

\subsection{Lifting}
As in the previou section, let us apply the lifting approach by considering
$$F^\dagger\eqdef u^\dagger\otimes q^\dagger\in\HDLD.$$
By the proof of \Cref{lemma isomorphism B2 and Bochner}, it is evident that $\HDLD$ is isometrically isomorphic to $\LDHD$ via the operator 
$$
\HDLD\ni a\otimes b\mapsto b\otimes a\in\LDHD.
$$
As a consequence, we could work with $q^{\dagger}\otimes u^{\dagger}\in\LDHD$ indifferently. We chose the former to be consistent with the boundary measurements setting of \Cref{sec_calderon}, in which this choice yields an arguably nicer space. 

Since $d\le3$, we have that $\HD$ is a reproducing kernel Hilbert space (see \citet{adams2003sobolev,brezisFunctionalAnalysisSobolev2011,novak2018reproducing}). Moreover, since $\Omega$ has Lispchitz boundary, we notice by \citet[Theorem 1.4.3.1]{Grisvard1986-ba} that $\HD=\restriction{H^2(\RR^d)}{\Omega}$. Therefore, we obtain that the kernel $K$ of $\HD$ coincides with the restriction to $\Omega\times\Omega$ of the kernel of $H^2(\RR^d)$ (see \citet[Theorem page 351]{aronszajn1950theory} or \citet[Theorem 6]{berlinet2011reproducing}). By the translation invariance of the kernel of $\HD$ \citep{novak2018reproducing}, we notice that \Cref{equation condition on K for restriction to diag} holds for $K$.
Thus, using \Cref{operator restriction to diag} with $m=0$ and $\calH=\HD$ and the isomorphism between $\LDHD$ and $\HDLD$, we can define $\restriction{F}{\diag}\in\LD$ for every $F\in\HDLD$.

With this definition, we see that $u^\dagger=u_{F^\dagger}$, where, given some $F\in\HDLD$, we denote by $u_F\in\HD$ the unique solution to 
\begin{equation*}\left\{
    \begin{aligned}
        \Delta u_F&=\restriction{F}{\operatorname{diag}(\Omega\times\Omega)}&&\mathrm{in}~\Omega,\\
        u_F&=f&&\mathrm{on}~\partial\Omega.
    \end{aligned}\right.
\end{equation*}

As in the previous section, we notice that the map $F\mapsto u_F$ is \emph{affine}, while $q\mapsto u$ is nonlinear. In order to obtain a \emph{linear} mapping from the unknown to the observations, we also define, for every $F\in\HDLD$, the function $v_F=u_F-\Tilde{f}\in H_0^1(\Omega)\cap\HD$, where $\Tilde{f}\in \HD$ is the harmonic extension of the boundary datum $f\in H^{3/2}(\partial\Omega)$, so that the map $F\mapsto v_F$ is linear.

To leverage the information that $F^\dagger(\cdot, y)=q^\dagger(y)u^\dagger$ is proportional to $u^\dagger$ for almost every\ $y\in\Omega$, we also observe that $F^\dagger$ satisfies
\begin{equation*}
    \int_{\Omega}F^\dagger(\cdot,y)dy=\bigg[\int_{\Omega}q^\dagger\bigg]u^\dagger.
\end{equation*}
As a result, assuming we know $\int_{\Omega}q^\dagger>0$, we set $\mathcal{H}\eqdef\left(\HcapH\right) \times H^2(\Omega)$, where we endow $\HcapH$ with the $L^2$-norm of the Laplacian (that is equivalent to the classical Sobolev norm), and define
\begin{equation}\label{eq:Phi-internal}
    \begin{aligned}
        \Phi \colon \HDLD  &\to  \mathcal{H}\\
        F &\mapsto \bigg(v_F, \int_\Omega F(\cdot,y)dy\bigg),
    \end{aligned}
\end{equation}
and $\m\eqdef (u^\dagger-\Tilde{f},[\int_{\Omega}q^\dagger]u^\dagger)=\Phi(F^\dagger)\in\mathcal{H}$. We propose to recover the unknown $F^\dagger$ by solving
\begin{equation}
        \mathrm{find}~F\in\HDLD~\mathrm{s.t.}~\Phi F=\m~\mathrm{and}~\mathrm{rank}(F)=1,
        \tag{$\mathcal{P}_{\mathrm{lifted}}$}
        \label{pb_lifted_internal_measurements}
\end{equation}
where we used the identification $\HDLD \cong \calB_2(H^2(\Omega);L^2(\Omega))$ (see \Cref{sec_bochner}) to define $\operatorname{rank} F$.
As in the previous section, our first result below proves the consistency of this lifting approach.
\begin{remark}\label{remark on the assumption integral q^dagger}
    The assumption that $\int_{\Omega}q^\dagger$ is known may appear artificial. It is required in the proof of \Cref{lemma_consistency_internal} to remove the invariance of the lifted variable $F=w\otimes p$ to the multiplication of $w$ and $p$ by a nonzero constant $\mu$ and $1/\mu$, respectively. A possible way to remove this assumption would be to introduce an additional variable $\alpha\in\RR$ and impose the constraint $\int_{\Omega}F^\dagger(\cdot,y)dy=\alpha u^\dagger$. Reasoning as in the proof of \Cref{lemma_consistency_internal}, one can show that the unique solution of this alternative lifted problem is $(F,\alpha)=(F^\dagger,\int_{\Omega}q^\dagger)$. As the introduction of this additional variable makes the construction of a dual certificate more difficult, we leave the investigation of this approach to future works.
\end{remark}
\begin{lemma}
    The unique solution to \Cref{pb_lifted_internal_measurements} is $F^\dagger=u^\dagger\otimes q^\dagger \in \HDLD$.
    \label{lemma_consistency_internal}
\end{lemma}
\begin{proof}
    Let $F\in \HDLD$ be a solution to \cref{pb_lifted_internal_measurements}. Since $\operatorname{rank}(F)=1$ we have that $F=w\otimes p$ for some $w\in H^2(\Omega)$, $p \in L^2(\Omega)$. The equality $\Phi F=\m$ yields
    \begin{align}\label{eq: phi1=b1 internal}
    &v_F=v_{F^\dagger}, \\
    \label{eq: phi2=b2 internal}
    &\bigg[\int_\Omega p\bigg] w = \bigg[\int_\Omega q^\dagger\bigg] u^\dagger. 
    \end{align}
    Since the right hand side of \Cref{eq: phi2=b2 internal} is non-zero, its left hand side is also non-zero. Hence, we obtain that 
    \begin{equation}\label{eq: rel w e u}
        w=\frac{\int_\Omega q^\dagger}{\int_\Omega p}u^\dagger.
    \end{equation}
    On the other hand, \Cref{eq: phi1=b1 internal} gives $\restriction{F}{\diag}=\restriction{F^\dagger}{\diag}$, that is to say
    \begin{equation}\label{eq: rel pw qu}
        w(x)p(x)=u^\dagger(x)q^\dagger(x), \qquad \text{a.e.\ } x\in\Omega.
    \end{equation}
    Plugging \Cref{eq: rel w e u} into \Cref{eq: rel pw qu} we obtain
    \begin{equation*}
        p=\frac{\int_\Omega p}{\int_\Omega q^\dagger}q^\dagger.
    \end{equation*}
    Therefore, we can conclude that
    \begin{equation*}
        F=w\otimes p=\frac{\int_\Omega q^\dagger}{\int_\Omega p}u^\dagger\otimes \frac{\int_\Omega p}{\int_\Omega q^\dagger}q^\dagger=u^\dagger\otimes q^\dagger =F^\dagger.
    \end{equation*}
    This concludes the proof.
\end{proof}

We have therefore reduced the problem of recovering $q^\dagger$ from $u^\dagger$ to the resolution of a \emph{linear} system, $\Phi F=\m$, under a non-convex rank-one constraint. In the next section we consider the relaxed problem where we convexify this constraint as in \Cref{sec_convex_relaxation_calderon} to finally obtain a fully convex optimization problem that allows us to recover $q^\dagger$.
\begin{remark}\label{remark linear system internal measurements}
Before turning to the convex relaxation of \cref{pb_lifted_internal_measurements}, we notice that this inverse problem can in fact be solved by considering a lifted linear system of equations, without relying on nuclear norm minimization as described in \Cref{sec_convex_relaxation} below. Indeed, considering the bounded linear operator
    \begin{equation*}
    \begin{aligned}
        \Psi \colon L^2(\Omega)\times \HDLD  &\to \left(\HcapH\right) \times \HDLD \\
        (p,F) &\mapsto (v_F, F-u^\dagger\otimes p),
    \end{aligned}
\end{equation*}
we have that $(q^\dagger,F^\dagger)$ is the only pair in $L^2(\Omega)\times\HDLD$ satisfying $\Psi(p,F)=(v_{F^\dagger},0)$. To see this, we notice that $v_F=v_{F^\dagger}$ implies $\restriction{F}{\diag}=\restriction{F^\dagger}{\diag}$. Therefore, we get that
    \begin{equation*}
u^\dagger(x)p(x)=F(x,x)=F^\dagger(x,x)=u^\dagger(x)q^\dagger(x), \qquad \text{a.e.\ } x\in\Omega,
    \end{equation*}
    which yields $p(x)=q^\dagger(x)$. Finally, we can conclude that $F=F^\dagger$ by observing that $0=F-u^\dagger\otimes p=F-F^\dagger$. However, we stress that the goal of this section is to prove the validity of the non-degenerate source condition for a similar, but easier, nonlinear inverse problem than the Calder\'on one. Therefore, we face this problem via nuclear norm minimization techniques.
\end{remark}
\subsection{Convex relaxation}
\label{sec_convex_relaxation}
Let $\sigma\eqdef\|u^\dagger\|_{H^2(\Omega)}\|q^\dagger\|_{L^2(\Omega)}>0$.
From now on, we denote the normalized versions of $u^\dagger$ and $q^\dagger$ by $u$ and $q$, namely, $u=u^\dagger/\|u^\dagger\|_{H^2(\Omega)}$ and $q=q^\dagger/\|q^\dagger\|_{L^2(\Omega)}$, so that $F^\dagger=\sigma u\otimes q$.

As previously discussed, a natural convex proxy for the rank is the nuclear norm. This naturally leads us to consider the following relaxed convex problem
\begin{equation}
        \underset{F\in \HDLD}{\mathrm{min}}~\|
        F\|_*~~\mathrm{s.t.}~~\Phi F=\m.
        \tag{$\mathcal{P}_{\mathrm{relaxed}}$}
        \label{primal_internal}
\end{equation}

\subsubsection{Exact and robust recovery}
We now show that the non-degenerate source condition implies an exact recovery result for the unknown $q^\dagger$. In the following section, we show that, under suitable assumptions on $q^\dagger$, this condition indeed holds.
\begin{proposition}
    If the non-degenerate source condition
    \begin{equation}
        \mathrm{there~exists~} H\in\mathrm{Im}(\Phi^*)~\mathrm{s.t.}~H=u\otimes q+W~~\mathrm{with}~Wu=0,~W^*q=0~\mathrm{and}~\|W\|<1
        \tag{NDSC}
        \label{ndsc_internal}
    \end{equation}
    holds, then $F^\dagger=\sigma u\otimes q=u^\dagger\otimes q^\dagger$ is the unique solution to \eqref{primal_internal}. Moreover, it holds that \begin{equation*}
        q^\dagger=\frac{1}{\int_{\partial\Omega} \lvert f\rvert^2}\int_{\partial\Omega}\restriction{F^\dagger}{\partial\Omega}(x,\cdot)f(x)dx.
    \end{equation*}
\end{proposition}
\begin{proof}
The first part of the statement is an immediate consequence of  \Cref{prop_exact_recovery_op} with $N=1$.   In order to show the last part of the statement, it is sufficient to observe that $F^\dagger(x,y)=f(x)q^\dagger(y)$ for almost every $(x,y)\in\partial\Omega\times\Omega$.
\end{proof}

We now turn to the question of robust recovery. We assume that we only know   $u^\delta\in\HD$ such that $\|u^\delta-u^\dagger\|_{\HD}\leq\delta$, for some $\delta>0$. As a result, we obtain that $$\|\m^{\delta}-\m\|_{\mathcal{H}}\leq \delta\sqrt{1+\bigg[\int_{\Omega}q^\dagger\bigg]^2}.$$
We wish to estimate the unknown $F^\dagger$ by solving
\begin{equation}
        \underset{F\in \HDLD}{\mathrm{min}}~\frac{1}{2}\|\Phi F-\m^{\delta}\|_{\mathcal{H}}^2+\lambda\|F\|_*
        \tag{$\mathcal{P}^{\lambda,\delta}_{\mathrm{relaxed}}$}
        \label{primal_noisy_internal}
\end{equation}
for some parameter $\lambda>0$. We have the following robust recovery result.
\begin{proposition}
   Let $\delta$ and $c$ be two positive constants. If the non-degenerate source condition \cref{ndsc_internal} holds with $H=\Phi^*p$ then, for any minimizer $F^\delta$ of \cref{primal_noisy_internal} with $\lambda =c\delta$ and $u^{\delta}\in\HD$ such that $\|u^{\delta}-u^\dagger\|_{\HD}\leq \delta$, we have that $\|F^{\delta}-F^\dagger\|_{\HDLD}=\mathcal{O}(\delta)$. Moreover, by setting $$
        q^\delta=\frac{1}{\int_{\partial\Omega} \lvert f\rvert^2}\int_{\partial\Omega}\restriction{F^\delta}{\partial\Omega}(x,\cdot)f(x)dx,$$ it follows that $\|q^\delta-q^\dagger\|_{L^2(\Omega)}=\mathcal{O}(\delta)$.
   \label{robust_recovery_internal} 
\end{proposition}
\begin{proof}
The first part of the statement follows by
 \Cref{prop_robust_recovery_op} and by the fact that $\restriction{\Phi}{T}$ is injective and has closed range (\Cref{lemma_apriori_estimate} below).
 The second part of the statement can be easily proved analogously to the corresponding estimate in \Cref{robust_recovery infinite}.
\end{proof}

We now turn to the proof that $\restriction{\Phi}{T}$ is injective and has closed range. We achieve this by proving an a priori estimate in \Cref{lemma_apriori_estimate}. Although \Cref{lemma_apriori_estimate} implies that $\restriction{\Phi}{T}$ is injective, we first give a proof of the latter for the sake of clarity. Then, the proof of \Cref{lemma_apriori_estimate} consists in making the same reasoning quantitative. 
\begin{lemma}\label{lemma injectivity of Phi_T}
    The operator $\Phi$ defined in \eqref{eq:Phi-internal} is injective on $T=\{u\otimes a+b\otimes q:(a,b)\in\ L^2(\Omega)\times H^2(\Omega)\}$.
\end{lemma}
\begin{proof}
A function $F=u\otimes a+ b\otimes q \in T$ belongs to the kernel of $\Phi$ if and only if 
\begin{align}
\label{kernel_cond_1}
    &u(x)a(x)+b(x)q(x)=0 \qquad \text{a.e. } x \in \Omega,  \\
    &\bigg[\int_\Omega q\bigg] b=-\bigg[\int_\Omega a\bigg] u.
    \label{kernel_cond_2}
\end{align}
From \eqref{kernel_cond_2}, we obtain $b=-[\int_\Omega a/\int_\Omega q]u$. Using \eqref{kernel_cond_1} and the positivity of $u$ in $\Omega$ (by the maximum principle \citet[Theorem~9.6]{Gilbarg2001-ry}), we get $a=[\int_\Omega a/\int_\Omega q]q$. By setting $\alpha\eqdef\int_\Omega a/\int_\Omega q$, we deduce that
\begin{equation*}
    F(x,y)=u(x)a(y)+b(x)q(y)=\alpha u(x)q(y)-\alpha u(x)q(y)=0 \qquad x\in\Omega,~\text{a.e. }y\in\Omega.
\end{equation*}
This concludes the proof.
\end{proof}
\begin{lemma}
    For every $F\in T$, we have
    \begin{equation}
        \|F\|_{\HDLD}\leq C_{\Phi}\|\Phi F\|_{\mathcal{H}},
        \label{apriori_estimate}
    \end{equation}
    where
    \begin{equation}
C_\Phi\eqdef\sqrt{2}\max\bigg(\frac{\|u^\dagger\|_{\HD}}{\inf u^\dagger},\frac{1}{\int_\Omega q^\dagger}\bigg(\|q^\dagger\|_{\LD}+\frac{\|q^\dagger\|_{L^\infty(\Omega)}\|u^\dagger\|_{\HD}}{\inf u^\dagger}\bigg)\bigg).
        \label{def_cphi}
    \end{equation}
    \label{lemma_apriori_estimate}
\end{lemma}
\begin{proof}
We define
\begin{equation}\label{eq expression Phi1 Phi2 internal}
    \begin{aligned}
        \Phi_1 \colon   \HDLD&\to\HcapH  \qquad\qquad& \Phi_2 \colon\HDLD&\to H^2(\Omega)\\
        F&\mapsto v_F \qquad\qquad& F&\mapsto \int_\Omega F(\cdot,y)dy.
    \end{aligned}
\end{equation}
so that $\Phi F=(\Phi_1F, \Phi_2F).$ We also recall that we endow the space $\HcapH$ with the $L^2$-norm of the Laplacian.

Now, consider $F=u\otimes a+b\otimes q\in T$ and set $\alpha\eqdef\int_\Omega a/\int_\Omega q$. Notice that $\Phi_2 F=(\int_\Omega q)(b+\alpha u)$, which yields
    \begin{equation}\label{eq bound b+alpha u internal}
        \|b+\alpha u\|_{\HD}= \frac{1}{\int_\Omega q}\|\Phi_2 F\|_{\HD}.
    \end{equation}
    Now, since $\Phi_1 F=v_F$, it holds that 
    \begin{equation*}
        \begin{aligned}
            \|\Phi_1F\|_{\HcapH}
            &=\|\Delta v_F\|_{\LD}=\|ua+bq\|_{\LD}=\|u(a-\alpha q)+(b+\alpha u)q\|_{\LD}\\
            &\geq (\inf u) \|a-\alpha q\|_{\LD}-\|q\|_{L^\infty(\Omega)}\|b+\alpha u\|_{\LD}.
        \end{aligned}
    \end{equation*}
    Using \Cref{eq bound b+alpha u internal}, we get
    \begin{equation}\label{eq bound a-alpha q internal}
        \|a-\alpha q\|_{\LD}\leq\frac{1}{\inf u}\bigg(\|\Phi_1 F\|_{\HcapH}+\frac{\|q\|_{L^\infty(\Omega)}}{\int_\Omega q}\|\Phi_2 F\|_{\HD}\bigg).
    \end{equation}
    Therefore, recalling that $\|u\|_{H^2(\Omega)}=\|q\|_{L^2(\Omega)}=1$, by \Cref{eq bound b+alpha u internal} and \Cref{eq bound a-alpha q internal} we obtain
    \begin{equation*}
    \begin{aligned}
        \|F\|_{\HDLD}
        &=\|u\otimes (a-\alpha q)+(b+\alpha u)\otimes q\|_{\HDLD}\leq \|a-\alpha q\|_{\LD}+\|b+\alpha u\|_{\HD}\\
        &\leq\frac{1}{\inf u}\bigg(\|\Phi_1 F\|_{\HcapH}+\frac{\|q\|_{L^\infty(\Omega)}}{\int_\Omega q}\|\Phi_2 F\|_{\HD}\bigg)+\frac{1}{\int_\Omega q}\|\Phi_2 F\|_{\HD}\\
        &\leq\sqrt{2}\max\bigg(\frac{1}{\inf u},\frac{1}{\int_\Omega q}\bigg(1+\frac{\|q\|_{L^\infty(\Omega)}}{\inf u}\bigg)\bigg)\|\Phi F\|_{\HcapH\times\HD}.
    \end{aligned}
    \end{equation*}
    Writing the constant in terms of $q^\dagger$ and $u^\dagger$ we obtain \eqref{apriori_estimate}.
\end{proof}

\subsubsection{Dual certificate}\label{subsection dual certificate internal measurements}
In this section, we show that the non-degenerate source condition \cref{ndsc_internal} holds under the assumption that the unknown $q^\dagger$ does not vary too much. The main result we prove is the following.
\begin{proposition}\label{proposition condition on q for existence dual certificate internal measurements}
    If
    \begin{equation}
        \frac{|\Omega|}{2\inf u}\frac{\sup q - \inf q}{\big[\int_\Omega q\big]^2}<1
        \label{ineq_suff_ndsc}
    \end{equation}
    then the non-degenerate source condition \cref{ndsc_internal} holds.
    \label{suff_unique_recovery_internal}
\end{proposition}
By \Cref{lemma positivity of u dagger} below, we have $\inf u \ge c(\Omega,\sup_\Omega q^\dagger,\inf_{\partial\Omega} f)/\|u^\dagger\|_{\HD}>0$, and so \eqref{ineq_suff_ndsc} holds beyond the trivial case of constant unknowns $q$.
For example, taking $\Omega=(0,1)$ and $q=1+q_0\mathbbm{1}_{(0.4,0.6)}$ with $f\equiv 1$, solving \cref{ip_internal_meas} numerically shows that \eqref{ineq_suff_ndsc} holds for every $q_0\in (-0.7, 0.8)$. Thus, this condition allows for moderate, and not necessarily very small, perturbations.
\begin{remark}
    An equivalent condition to \eqref{ineq_suff_ndsc} that depends more explicitly on the original unknown $q^\dagger$ (and not on its normalized version $q=q^\dagger/\|q^\dagger\|_{\LD}$) is the following:
    \begin{equation}
        \frac{|\Omega|}{2 \inf u^\dagger}\frac{\sup q^\dagger - \inf q^\dagger}{\big[\int_{\Omega}q^\dagger\big]^2}\|q^\dagger\|_{\LD}\|u^\dagger\|_{\HD}<1.
        \label{ineq_suff_ndsc_unnormalized}
    \end{equation}
    The dependence of  $\|u^\dagger\|_{\HD}$ on $q^\dagger$ can also be made more explicit. For this purpose, let $a_{\Omega}$ denote the norm of the harmonic extension operator from $H^{3/2}(\partial\Omega)$ to $H^2(\Omega)$, ${b_{\Omega}}$ be  the constant that appears in the elliptic regularity for the Poisson equation on $\Omega$, namely, ${b_{\Omega}\eqdef\mathrm{sup}~\{\|v\|_{\HD}:v\in H^1_0(\Omega),~\|\Delta v\|_{\LD}\leq 1\}}$, and $c_{\Omega}$ be the Poincaré constant of $\Omega$. Finally, defining $d_{\Omega}\eqdef a_{\Omega}\,\mathrm{max}(1,b_{\Omega}/2,b_{\Omega}c_{\Omega}^2)$, we claim  that the following condition implies \eqref{ineq_suff_ndsc_unnormalized}
    \begin{equation}
        \frac{|\Omega|}{2 \inf u^\dagger}\frac{\sup q^\dagger - \inf q^\dagger}{\big[\int_{\Omega}q^\dagger\big]^2}\|q^\dagger\|_{\LD}(1+\sup q^\dagger)^2 \|f\|_{H^{3/2}(\partial\Omega)}d_{\Omega}<1.
        \label{ineq_suff_ndsc_final}
    \end{equation}
    The proof of this fact is at the end of this section.
\end{remark}
It is worth observing that  $\inf u^\dagger\ge c>0$, where $c$ depends only on $\Omega$, $f$ and $\sup q^\dagger$.
\begin{lemma}
    Let $Q\eqdef \sup_{\Omega}q^\dagger>0$ and $f_0\eqdef\inf_{\partial\Omega}f>0$. There exists $c=c(\Omega,Q,f_0)>0$ such that $u^\dagger\geq c>0$ in $\overline{\Omega}$.
    \label{lemma positivity of u dagger}
\end{lemma}
\begin{proof}
   Let $u_Q\in C^2(\Omega)\cap C(\overline{\Omega})$ the solution to $-\Delta u_Q+Qu_Q=0$ in $\Omega$ and $u_Q=f_0$ on $\partial\Omega$. In particular, since $u_Q$ is not constant, the strong maximum principle for classical solutions \cite[Theorem 3.5]{Gilbarg2001-ry} ensures that $c\eqdef\min_{\overline{\Omega}} u_Q>0$. Now, notice that 
    \[
        (-\Delta+Q)(u^\dagger-u_Q)=-\Delta u^\dagger+Qu^\dagger=(Q-q)u^\dagger\ge0,\quad\text{by the positivity of $u^\dagger$,}
    \]
    and that $u^\dagger\ge u_Q$ on $\partial\Omega$.
    Hence, it follows from \citet[Theorem 9.6]{Gilbarg2001-ry} that $u^\dagger\ge u_Q$ in $\overline{\Omega}$, which ensures
    \(
        \inf_{\overline{\Omega}} u^\dagger\ge c>0.
    \)
\end{proof}
For the sake of clarity, we split the proof of \Cref{suff_unique_recovery_internal} into different lemmas. We first need an expression of the adjoint of the measurement operator $\Phi$.
\begin{lemma}
    It holds that
    \begin{equation*}
    \begin{aligned}
        \Phi^* \colon\HcapH \times H^2(\Omega) &\to\HDLD\\
        (v,w) &\mapsto [x\mapsto w(x)+\Delta v(\cdot)K(x,\cdot)].
    \end{aligned}
\end{equation*}
    \label{lemma expression of Phi* internal}
\end{lemma}
\begin{proof}
    Let us define $\Phi_1$ and $\Phi_2$ as in \Cref{eq expression Phi1 Phi2 internal} so that $\Phi F=(\Phi_1 F,\Phi_2 F)$ for every $F\in\HDLD$. 

    For every $v\in\HcapH$ we have $\Phi_1^*v\colon x\mapsto\Delta v(\cdot) K(x,\cdot)$, where $K$ is the reproducing kernel of $H^2(\Omega)$.
    Indeed, we have: 
    \begin{equation*}
        \begin{aligned}
            \langle v_F,v\rangle_{\HcapH}&=\int_\Omega \Delta v_F(y) \Delta v(y)dy\\
            &=\int_\Omega \restriction{F}{\diag}(y) \Delta v(y)dy\\
            &=\int_\Omega\langle F(\cdot,y),K(\cdot,y)\rangle_{H^2(\Omega)}\Delta v(y)dy\\
            &=\int_{\Omega}\bigg[\sum_{|\alpha|\leq 2}\int_{\Omega} (\mathrm{D}^{\alpha} F(x))(y) (\mathrm{D}^{\alpha}H(x))(y) dx\bigg]dy\\
            &=\sum_{|\alpha|\leq 2}\int_{\Omega}\langle\mathrm{D}^{\alpha}F(x),\mathrm{D}^{\alpha} H(x)\rangle_{\LD}dx\\
            &=\langle F,H\rangle_{\HDLD},
        \end{aligned}
    \end{equation*}
    with $H\colon x\mapsto \Delta v(\cdot)K(x,\cdot)$.

    On the other hand, for every $w\in H^2(\Omega)$, we have
    \begin{equation*}
        \begin{aligned}
            \langle \Phi_2 F,w\rangle_{H^2(\Omega)}
            &=\sum_{|\alpha|\leq 2}\int_{\Omega}\mathrm{D}^{\alpha}(\Phi_2 F)(x)\mathrm{D}^{\alpha}w(x)dx\\
            &=\sum_{|\alpha|\leq 2}\int_{\Omega}\bigg[\int_{\Omega}\mathrm{D}^{\alpha}F(x,y)\mathrm{D}^{\alpha}w(x)dy\bigg]dx\\
            &=\langle F,w\otimes \mathbbm{1}\rangle_{\HDLD}.
        \end{aligned}
    \end{equation*}
    As a result, we have $\Phi_2^*w=w\otimes\mathbbm{1}$. By noticing that $\Phi^*(v,w)=\Phi_1^*(v)+\Phi^*_2(w)$ for every $(v,w)\in\HcapH\times\HD$, we conclude the proof.
\end{proof}
As a second step, we characterize the set of the \emph{pre-certificates}, namely the set of the functions $H$ in the range of $\Phi^*$ whose projections on $T$ coincide with $u\otimes q$, i.e.\ such that $Hu=q$ and $H^*q=u$.
\begin{lemma}
    A function $H\in\operatorname{Im}(\Phi^*)$ satisfies $Hu=q$ and $H^*q=u$ if and only if there exists $\alpha\in\RR$ such that the following equality holds in $\LD$ for almost every $x\in\Omega$:
    \begin{equation*}
        H(x)=\frac{1}{\int_\Omega q}u(x)+\frac{q-\alpha}{u}K_x-\frac{1}{\int_\Omega q}\int_\Omega \frac{q(y)-\alpha}{u(y)}q(y)K_y(x)dy.
    \end{equation*}
    \label{lemma expression of pre-certificate internal}
\end{lemma}
\begin{proof}
Let $v\in\HcapH$, $w\in \HD$ and set $H=\Phi^*(v,w)$. The relations $Hu=q$ and $H^*q=u$ yield
\begin{align}
    \label{eq A*_hu=q rewritten}
    &q(y)=\langle w,u\rangle_{H^2(\Omega)}+\Delta v(y)u(y),\qquad y\in\Omega,\\
\label{eq A_hq=u rewritten}
    &u=\bigg[\int_\Omega q\bigg] w+\int_\Omega q(y)\Delta v(y)K_y dy.
\end{align}
Set $\alpha\eqdef\langle w,u\rangle_{H^2(\Omega)}$. By \Cref{eq A*_hu=q rewritten}, using the positivity of $u$, we obtain
\begin{equation*}
    \Delta v(y)=\frac{q(y)-\alpha}{u(y)},\qquad y\in\Omega.
\end{equation*}
Plugging this expression into \Cref{eq A_hq=u rewritten}, we obtain
\begin{equation*}
    w=\frac{1}{\int_\Omega q}u-\frac{1}{\int_\Omega q}\int_\Omega \frac{q(y)-\alpha}{u(y)}q(y)K_y dy.
\end{equation*}
Using that $H=\Phi^*(v,w)$ and the expression of $\Phi^*$ provided in \Cref{lemma expression of Phi* internal}, we conclude the proof.
\end{proof}
Finally, we can proceed with the proof of \Cref{proposition condition on q for existence dual certificate internal measurements}.
\begin{proof}[Proof of \Cref{proposition condition on q for existence dual certificate internal measurements}]
Using the expression of $H$ showed in \Cref{lemma expression of pre-certificate internal}, the goal of this proof is to find $\alpha$ such that $\langle Ha,b\rangle_{\HD}<1$ for every pair of functions $(a,b)\in\HD\times \LD$ with unit norm respectively orthogonal to $u$ and $q$. We have
\begin{equation*}
        H^*b=\frac{\int_\Omega b}{\int_\Omega q}u+\int_\Omega \bigg(b(y)-\frac{\int_\Omega b}{\int_\Omega q}q(y)\bigg)\frac{q(y)-\alpha}{u(y)}K_y dy
\end{equation*}
and, using that $\langle a,u\rangle_{H^2(\Omega)}=0$ and $\langle a,K_y\rangle_{H^2(\Omega)}=a(y)$,
\begin{equation*}
    \langle Ha,b\rangle_{\LD}=\langle a,H^*b\rangle_{\HD}=\int_\Omega\bigg(b(y)-\frac{\int_\Omega b}{\int_\Omega q}q(y)\bigg)\frac{q(y)-\alpha}{u(y)}a(y) dy.
\end{equation*}
Hence, we want to prove that the quantity
\begin{equation}
    \begin{aligned}
        &\sup\limits_{(a,b)\in \HD\times \LD} & & L_{\alpha}(a,b)\eqdef\int_\Omega\bigg(b(y)-\frac{\int_\Omega b}{\int_\Omega q}q(y)\bigg)\frac{q(y)-\alpha}{u(y)}a(y) dy\\
        & \mathrm{s.t.} & & \langle a,u\rangle_{\HD}=\langle b,q\rangle_{\LD}=0,~\|a\|_{\HD}=\|b\|_{\LD}=1
    \end{aligned}
    \label{sup_quantity_dual_certif}
\end{equation}
is strictly smaller than $1$ for some $\alpha\in\RR$. Using that $\|a\|_{\LD}\leq \|a\|_{\HD}=1$, we get
\begin{equation}
    |L_{\alpha}(a,b)|\leq \bigg\|\frac{q-\alpha}{u}\bigg\|_{L^\infty(\Omega)}\bigg\|\bigg(b-\frac{\int_\Omega b}{\int_\Omega q}q\bigg)a\bigg\|_{L^1(\Omega)}\leq \bigg\|\frac{q-\alpha}{u}\bigg\|_{L^\infty(\Omega)}\bigg\|b-\frac{\int_\Omega b}{\int_\Omega q}q\bigg\|_{L^2(\Omega)}.
    \label{ineq_l_alpha_a_b}
\end{equation}
Then, since $b$ and $q$ are orthogonal in $L^2(\Omega)$ and $\|q\|_{\LD}=\|b\|_{\LD}=1$, we obtain
\begin{equation*}
    \bigg\|b-\frac{\int_\Omega b}{\int_\Omega q}q\bigg\|^2_{L^2(\Omega)}=\bigg(\frac{\int_\Omega b}{\int_\Omega q}\bigg)^2 +1. 
\end{equation*}

Let $\beta=\int_{\Omega}q$. Using the orthogonality of $b$ and $q$ and the fact that $\|b\|_{\LD}=1$, we obtain
\begin{equation*}
    \bigg|\int_{\Omega}b\bigg|^2=\bigg|\int_{\Omega}b(1-\beta q)\bigg|^2\leq \|1-\beta q\|_{\LD}^2=|\Omega|-2\beta^2+\beta^2 \|q\|_{\LD}=|\Omega|-\beta^2,
\end{equation*}
so that
\begin{equation*}
    |L_{\alpha}(a,b)|\leq \frac{|\Omega|}{\big[\int_\Omega q\big]^2}\frac{\|q-\alpha\|_{L^\infty(\Omega)}}{\inf u}.
\end{equation*}
Using that $\|q-\alpha\|_{L^{\infty}(\Omega)}$ is minimal for $\alpha=(\inf q + \sup q)/2$ and that the associated minimum value is $(\sup q^\dagger-\inf q^\dagger)/2$, we get
\begin{equation*}
    |L_{\alpha}(a,b)|\leq\|q^\dagger\|_{\LD}\|u^\dagger\|_{\HD}\frac{|\Omega|}{\big[\int_{\Omega}q^\dagger\big]^2}\frac{\sup q^\dagger - \inf q^\dagger}{2 \inf u^\dagger},
\end{equation*}
where we used $q=q^\dagger/\|q^\dagger\|_{\LD}$ and $u=u^\dagger/\|u^\dagger\|_{\HD}$. 
\end{proof}

\begin{proof}[Proof that \eqref{ineq_suff_ndsc_final} implies \eqref{ineq_suff_ndsc_unnormalized}]
We recall that $a_{\Omega}$ is the norm of the harmonic extension operator from $H^{3/2}(\partial\Omega)$ to $H^2(\Omega)$, $b_{\Omega}=\mathrm{sup}~\{\|v\|_{\HD}\,\rvert\,v\in H^1_0(\Omega),~\|\Delta v\|_{\LD}\leq 1\}$, $c_{\Omega}$ the Poincaré constant of $\Omega$ and $d_{\Omega}= a_{\Omega}\,\mathrm{max}(1,b_{\Omega}/2,b_{\Omega}c_{\Omega}^2)$.
We have to show that
\begin{equation}\label{eq_bound u^dagger internal}
    \|u^\dagger\|_{\HD}\leq d_\Omega \|f\|_{H^{3/2}(\partial\Omega)}(1+\sup q^\dagger)^2.
\end{equation}
To prove this, recalling that $u^\dagger=v^\dagger+\tilde f$, where $v^\dagger=v_{F^\dagger}$ and $\tilde f$ is the harmonic extension of $f$, we notice that $\|u^\dagger\|_{\HD}\leq \|v^\dagger\|_{\HD}+\|\tilde{f}\|_{\HD}$. Then, we observe that $$\|v^\dagger\|_{\HD}\leq b_{\Omega}\|\Delta v^\dagger\|_{\LD}=b_{\Omega}\|q^\dagger v^\dagger+q^\dagger\tilde{f}\|_{\LD}\leq b_{\Omega}\|q^\dagger\|_{\infty}(\|v^\dagger\|_{\LD}+\|\tilde{f}\|_{\LD}).$$
Therefore, we get
\begin{align}\label{eq_inequality for u^dagger proposition dual certificate internal}
    \|u^\dagger\|_{\HD}
    &\leq b_{\Omega}\|q^\dagger\|_{\infty}(\|v^\dagger\|_{\LD}+\|\tilde{f}\|_{\LD})+\|\tilde{f}\|_{\HD}\\
    &\leq b_{\Omega}\|q^\dagger\|_{\infty}\|v^\dagger\|_{\LD}+a_{\Omega}(1+b_{\Omega}\|q^\dagger\|_{\infty})\|f\|_{H^{3/2}(\partial\Omega)}.
\end{align}    
By \eqref{ip_internal_meas}, we obtain that $-\Delta v^\dagger+q^\dagger v^\dagger=-q^\dagger \tilde{f}$. By testing this PDE with $v^\dagger$, we get
$$
\|v^\dagger\|_{\LD}^2\leq c^2_{\Omega} \|\nabla v^\dagger\|_{\LD}^2 \leq c^2_{\Omega}\bigg(\|\nabla v^\dagger\|_{\LD}^2+\int_{\Omega}q^\dagger {v^\dagger}^2\bigg)=-c^2_{\Omega} \int_{\Omega} q^\dagger\tilde{f}v^\dagger\leq c^2_{\Omega}\|q^\dagger\|_{\infty}\|\tilde{f}\|_{\LD}\|v^\dagger\|_{\LD},
$$ 
which yields $\|v^\dagger\|_{\LD}\leq c^2_{\Omega}\|q^\dagger\|_{\infty}\|\tilde{f}\|_{\LD}$. Plugging this last inequality in \Cref{eq_inequality for u^dagger proposition dual certificate internal}, we obtain \Cref{eq_bound u^dagger internal} and this conclude the proof.
\end{proof}

\begin{remark}
    We stress that our bound on $L_{\alpha}(a,b)$ does not seem sharp and could be improved. This could possibly lead to significantly less restrictive conditions on $q^\dagger$. In particular, our use of the H\"older inequality twice in \eqref{ineq_l_alpha_a_b} is  suboptimal.
\end{remark}
\section{Conclusion}\label{sec:conclusion}
\paragraph{Summary.} In this work, we explored the adaptation of convex lifting techniques to the Calder\'on problem. We showed that these techniques can be generalized to allow for the recovery of rank-one operators between Hilbert spaces. Considering a toy inverse problem in which internal measurements are available, we were able to show that the unwknown coefficient, under suitable assumptions, is the unique solution of a convex optimization problem. As reconstruction methods in inverse problems for PDEs usually suffer from the problem of local convergence, this property is highly interesting, given that convex optimization problems can be solved with globally convergent algorithms. In the case of boundary measurements, we proved that a non-degenerate source condition is sufficient to obtain the same result and left the investigation of its validity for future works.

\paragraph{Numerical perspectives.} Although they were not discussed in this article, there are numerous numerical perspectives to this work. First, the numerical resolution of \eqref{primal_calderon} and its regularized counterpart should be investigated. As our estimate of the unknown is defined as one of their solutions, obtaining a reconstruction method that is implementable in practice requires reliable solvers for these problems. Furthermore, the lifting used leads to spaces of functions of several variables, and it would be interesting to investigate how the use of the Burer-Monteiro factorization (see \cite{burer-monteiro-2005,boumalDeterministicGuaranteesBurerMonteiro2020,waldspurger-waters-2020,lingLocalGeometryDetermines2025,endorBenignLandscapeBurerMonteiro2025}) may help reduce the computational complexity. Then, we stress that, in the literature on sparse or low complexity regularization, the verification of the non-degenerate source condition is known to be challenging. For this reason, several works proposed to numerically investigate its validity by computing the pre-certificate defined in \Cref{subsubsec_dual_certificate_infinite}. We think this perspective is particularly promising.

\paragraph{Identifiability from finitely many Cauchy data.}
As highlighted earlier, \Cref{proposition equivalence original prob and rank min prob infinite} strongly relies on identifiability results for potentials in the case where the DN-map is not well-defined, that is, when the Dirichlet problem does not necessarily admit a unique solution. In this case, the identifiability of potentials is well understood when the Cauchy data sets coincide; however, we are not aware of analogous results for finitely many measurements. Nevertheless, we have chosen to present the entire section using a finite number of measurements, as this approach is more directly oriented towards potential applications, leaving the question of identifiability for finitely many measurements in the case of Cauchy data sets for future work.

\paragraph{Removing the assumption that $\int_{\Omega}q^\dagger$ is known.} As underlined in \Cref{remark int on K,remark on the assumption integral q^dagger}, the assumption that the integral of the unknown coefficient is known (on which we rely in both \Cref{sec_internal_measurements,sec_calderon}) may appear artificial. The knowledge of this quantity is exploited to remove an invariance of the lifted variable. We think that finding a way to remove this assumption is highly relevant.
\section*{Acknowledgments}
The authors warmly thank Irène Waldspurger for insightful discussions about this work.

Co-funded by the European Union (ERC, SAMPDE, 101041040 and ERC, WOLF, 	101141361). Views and opinions
expressed are however those of the authors only and do not necessarily reflect those of the European
Union or the European Research Council. Neither the European Union nor the granting authority
can be held responsible for them. This work was supported by the “Gruppo Nazionale per l’Analisi Matematica, la Probabilità e le loro Applicazioni” (GNAMPA–INdAM), of which G.\ S.\ Alberti and S.\ Sanna are members.
The research was supported in part by the MIUR Excellence Department Project awarded to Dipartimento di Matematica, Università di Genova, CUP D33C23001110001. Co-funded by European Union – Next Generation EU, Missione 4 Componente 1 CUP D53D23005770006 and CUP D53D23016180001.
\bibliographystyle{apalike}
\bibliography{ref}
\appendix
\section{Strong duality}
\label{appendix_strong_duality_internal}
To prove \Cref{prop_strong_dual_op}, we rewrite \cref{dual_op} as
    \begin{equation}
        \underset{p\in \mathcal{H}}{\mathrm{inf}}~\mathcal{F}(p)+\mathcal{G}(\Phi^* p),
        \tag{$\mathcal{Q}$}
        \label{predual}
    \end{equation}
    with $\mathcal{F}=-\langle \cdot,\m\rangle_\calH$
    and $\mathcal{G}=\iota_{\{\|\cdot\|\leq 1\}^N}$. The functions $\mathcal{F}$ and $\mathcal{G}$ are convex proper and lower semi-continuous on $\mathcal{H}$ and $\mathcal{B}_2(\calH_1;\calH_2)^N$.
    Moreover, the infimum in \cref{predual} is finite since for every $p\in\calH$ with $\|(\Phi^* p)_i\|\le 1$ for every $1\leq i\leq N$, we have
    \begin{equation*}
        -\langle p,\m\rangle=-\langle p,\Phi F^\dagger\rangle =-\langle \Phi^*p,F^\dagger\rangle \geq -\sum\limits_{i=1}^N\|F^\dagger_i\|_* > -\infty.
    \end{equation*}
    Finally $\mathcal{F}$ is finite at $0$ and $\mathcal{G}$ is finite and continuous at $0=\Phi^*0$. By \citet[Theorem III.4.1]{ekeland1999convex}, we hence obtain that 
    \begin{equation*}
        \begin{aligned}
        \underset{F\in\calB_2(\calH_1;\calH_2)^N}{\mathrm{sup}}~-\mathcal{F}^*(\Phi F)-\mathcal{G}^*(-F)&=\underset{F\in \calB_2(\calH_1;\calH_2)^N}{\mathrm{sup}}~-\sum\limits_{i=1}^N\|-F_i\|_*~~\mathrm{s.t.}~~\Phi F=-\m\\
        &=-\underset{F\in \calB_2(\calH_1;\calH_2)^N}{\mathrm{inf}}~\sum\limits_{i=1}^N\|F_i\|_*~~\mathrm{s.t.}~~\Phi F=\m
        \end{aligned}
    \end{equation*}
    is attained and equal to the value of \cref{predual}, which is the opposite of the value
    of \cref{dual_op}. We hence have that \cref{primal_op} has a solution and that the values of
    \cref{primal_op} and \cref{dual_op} are equal. The second part of \Cref{prop_strong_dual_op} follows by applying \citet[Proposition III.4.1]{ekeland1999convex}.

\section{Robust recovery with linear rate}
We begin with a first lemma providing a bound on the orthogonal projection of the error on $T^\perp$. Its proof is directly adapted from \citet[Lemma 6.2]{vaiterLowComplexityRegularizations2014}. This bound depends on the Bregman divergence associated to the nuclear norm between $F^{\delta}$ and $F^\dagger$, which is defined for every $H$ satisfying $H_i\in\partial \|\cdot\|_*(F_i^\dagger)$ for every $1\leq i\leq N$ by
\begin{equation*}
    D_H(F,F^\dagger)\eqdef \sum\limits_{i=1}^N\|F_i\|_*-\|F_i^\dagger\|_*-\langle H_i,F_i-F_i^\dagger\rangle,\qquad F\in\calB_2(\calH_1;\calH_2)^N.
\end{equation*}
\label{appendix_robust_recovery}
\begin{lemma}\label{lemma Bound on P_Tperp with bregman divergence}
For every $H$ such that $H_i\in\partial \|\cdot\|_*(F_i^\dagger)$ for every $1\leq i\leq N$ we have
    \begin{equation}
        \sum\limits_{i=1}^N\|P_{T_i^\perp}((F^{\delta}-F^\dagger)_i)\|_{\HS}\leq \frac{D_H(F^{\delta},F^\dagger)}{1-\underset{1\leq i\leq N}{\mathrm{max}}\|W_i\|}
        \label{bound_orth_proj_bregman}
    \end{equation}
    with $W_i=H_i-u_i\otimes v_i$ ($1\leq i\leq N$).
\end{lemma}
\begin{proof}
    For every $V\in T^\perp$ (that is to say $V_iv_i=0$ and $V_i^*u_i=0$ for every $1\leq i\leq N$) such that $\|V_i\|\leq 1$ for every $1\leq i\leq N$ we have that $u_i\otimes v_i+V_i$ belongs to $\partial \|\cdot\|_*(F_i^\dagger)$. This yields
        \begin{equation*}
        \begin{aligned}
        D_H(F^\delta,F^\dagger)&\geq D_H(F^\delta,F^\dagger)-D_{(u_i\otimes v_i)_{1\leq i\leq N}+V}(F^\delta,F^\dagger)\\
        &=\sum\limits_{i=1}^N\langle u_i\otimes v_i+V_i-H_i,(F^\delta-F^\dagger)_i\rangle\\
        &=\sum\limits_{i=1}^N\langle V_i-W_i,(F^\delta-F^\dagger)_i\rangle.\\
        \end{aligned}
        \end{equation*}
    We claim that, for every $1\leq i\leq N$, it holds
    \begin{equation}
        \textstyle \mathrm{sup}\,\{\sum_{i=1}^N\langle V_i,(F^{\delta}-F^\dagger)_i\rangle\,\rvert\,V\in T^\perp,~\mathrm{max}_{1\leq i\leq N}\|V_i\|\leq 1\}=\sum_{i=1}^N\|P_{T_i^\perp}((F^\delta-F^\dagger)_i)\|_*.
    \end{equation}
    To see this, we first use \eqref{eq expression proj Tperp} to get that $P_{T_i^\perp}(Z_i)=P_{v_i^\perp}Z_i P_{u_i^\perp}$ for every $Z\in\calB_2(\calH_1;\calH_2)^N$ where $P_{u_i^\perp}$ and $P_{v_i^\perp}$ are the orthogonal projections on $\mathrm{Span}(\{u_i\})^\perp$ and $\mathrm{Span}(\{v_i\})^\perp$. Since these two projection operators are non-expansive, we obtain $||P_{T_i^\perp}(Z_i)||\leq \|Z_i\|$. Consequently, we see that for every $Z$ such that $\|Z_i\|\leq 1$ for every $1\leq i\leq N$, we have $\|P_{T_i^\perp}(Z_i)\|\leq 1$ and hence
    \begin{equation*}
    \begin{aligned}
    \textstyle\sum_{i=1}^N\langle P_{T_i^\perp}((F^{\delta}-F^\dagger)_i),Z_i\rangle&\textstyle=\sum_{i=1}^N\langle P_{T_i^\perp}((F^{\delta}-F^\dagger)_i),P_{T_i^\perp}(Z)\rangle\\
    &\textstyle\leq \mathrm{sup}~\{\sum_{i=1}^N\langle P_{T_i^\perp}((F^{\delta}-F^\dagger)_i),V_i\rangle\,\rvert\,V\in T^\perp,~\mathrm{max}_{1\leq i\leq N}\|V_i\|\leq 1\}\\
    &\textstyle=\mathrm{sup}~\{\sum_{i=1}^N\langle (F^{\delta}-F^\dagger)_i,V_i\rangle\,\rvert\,V\in T^\perp,~\mathrm{max}_{1\leq i\leq N}\|V_i\|\leq 1\}.
    \end{aligned}
    \end{equation*}
    Taking the supremum over all $Z$ such that $\|Z_i\|\leq 1$ for every $1\leq i\leq N$, we obtain
    \begin{equation*}
         \textstyle\mathrm{sup}~\{\sum_{i=1}^N\langle V_i,(F^{\delta}-F^\dagger)_i\rangle\,\rvert\,V\in T^\perp,~\mathrm{max}_{1\leq i\leq N}\|V_i\|\leq 1\}\geq\sum_{i=1}^N\|P_{T_i^\perp}((F^\delta-F^\dagger)_i)\|_*.
    \end{equation*}
    The reverse inequality is a direct consequence of the dual formulation of the nuclear norm \eqref{dual_formula_nuclear_norm}.
    
    As a result, we obtain
    \begin{equation*}
        \begin{aligned}
        D_H(F^\delta,F^\dagger)&\geq\sum\limits_{i=1}^N\|P_{T_i^\perp}((F^{\delta}-F^\dagger)_i)\|_* -\langle W_i,(F^\delta-F^\dagger)_i\rangle\\
        &=\sum\limits_{i=1}^N\|P_{T_i^\perp}((F^{\delta}-F^\dagger)_i)\|_* -\langle W_i,P_{T_i^\perp}((F^\delta-F^\dagger)_i)\rangle\\
            &\geq \bigg[\underset{1\leq i\leq N}{\max}(1-\|W_i\|)\bigg]\Bigg[\sum_{i=1}^N\|P_{T_i^\perp}((F^{\delta}-F^\dagger)_i)\|_*\Bigg]\\
            &\geq \bigg[\underset{1\leq i\leq N}{\max}(1-\|W_i\|)\bigg]\Bigg[\sum_{i=1}^N\|P_{T_i^\perp}((F^{\delta}-F^\dagger)_i)\|_{\HS}\Bigg].
        \end{aligned}
    \end{equation*}
\end{proof}
We can now proceed with the proof of \Cref{prop_robust_recovery_op}.
\begin{proof}[Proof of \Cref{prop_robust_recovery_op}]
We have
    \begin{equation*}
        \begin{aligned}
            \sum\limits_{i=1}^N\|(F^{\delta}-F^\dagger)_i\|_{\HS}&\leq \sum\limits_{i=1}^N\| P_{T_i}((F^{\delta}-F^\dagger)_i)\|_{\HS} + \|P_{T_i^\perp}((F^{\delta}-F^\dagger)_i)\|_{\HS}&&\\
            &\leq C_{\Phi}\|\Phi P_T(F^{\delta}-F^\dagger)\|_{\mathcal{H}}+\sum\limits_{i=1}^N\|P_{T_i^\perp}((F^{\delta}-F^\dagger)_i)\|_{\HS}&&\text{by \citet[Theorem 2.21]{brezisFunctionalAnalysisSobolev2011}},\\
            &\leq C_{\Phi}\|\Phi(F^{\delta}-F^\dagger)\|_{\mathcal{H}}+(1+C_{\Phi}\|\Phi\|)\sum\limits_{i=1}^N\|P_{T_i^\perp}((F^{\delta}-F^\dagger)_i)\|_{\HS}&&\\
            &\leq C_{\Phi}\|\Phi(F^{\delta}-F^\dagger)\|_{\mathcal{H}}+\frac{1+C_{\Phi}\|\Phi\|}{1-\underset{1\leq i\leq N}{\max}\|W_i\|}D_H(F^{\delta},F^\dagger)&&\text{by \eqref{bound_orth_proj_bregman}}.
        \end{aligned}
    \end{equation*}
    As a result, it remains only to bound $\|\Phi(F^{\delta}-F^\dagger)\|_{\mathcal{H}}$ and $D_H(F^{\delta},F^\dagger)$ in terms of $\delta$, which is the object of \Cref{lemma_bregman_rate_prediction} below.
\end{proof}
\begin{lemma}[Bregman rate and prediction error]\label{lemma Bregman rate and prediction error}
    Let $\delta$ and $c$ be two positive constants. Let $H=\Phi^*p$ be such that $H_i\in\partial \|\cdot\|_*(F_i^\dagger)$ for every $1\leq i\leq N$. Then, for any minimizer $F^\delta$ of \eqref{primal_noisy_op} with $\lambda =c\delta$ and $\m^{\delta}$ such that $\|\m^{\delta}-\m\|_{\mathcal{H}}\leq \delta$ we have
    \begin{equation*}
    \begin{aligned}
        D_H(F^\delta,F^\dagger)&\leq \frac{(1+c\|p\|_{\mathcal{H}})^2}{2c}\delta,\\
        \|\Phi F^{\delta}-\Phi F^\dagger\|_{\mathcal{H}}&\leq 2(1+c\|p\|_{\mathcal{H}})\delta.
    \end{aligned}
    \end{equation*}
    \label{lemma_bregman_rate_prediction}
\end{lemma}
\begin{proof}
The proof is directly adapted from \citet[Proposition 3.41]{scherzerVariationalMethodsImaging2008} and \citet[Lemma 6.1]{vaiterLowComplexityRegularizations2014}.
By optimality of $F^{\delta}$ and since $\|\Phi F^\dagger-\m^{\delta}\|_{\mathcal{H}}\leq \delta$ we have
\begin{equation}
\begin{aligned}
    \frac{1}{2}\|\Phi F^{\delta}-\m^{\delta}\|_{\mathcal{H}}^2+\lambda \sum_{i=1}^N\|F_i^{\delta}\|_*&\leq \frac{1}{2}\|\Phi F^\dagger-\m^{\delta}\|_{\mathcal{H}}^2+\lambda \sum_{i=1}^N\|F^\dagger_i\|_*\\
    &\leq \frac{\delta^2}{2}+\lambda \sum_{i=1}^N\|F^\dagger_i\|_*.
    \end{aligned}
    \label{ineq_optimality}
\end{equation}
Now, we have
\begin{equation*}
    D_H(F^{\delta},F^\dagger)=\sum_{i=1}^N\|F_i^{\delta}\|_*-\|F_i^\dagger\|_*-\langle H_i,(F^{\delta}-F^\dagger)_i\rangle=\Bigg[\sum_{i=1}^N\|F_i^{\delta}\|_*-\|F_i^\dagger\|_*\Bigg]-\langle p,\Phi F^{\delta}-\Phi F^\dagger\rangle_{\mathcal{H}}.
\end{equation*}
This yields
\begin{equation*}
\begin{aligned}
    D_H(F^{\delta},F^\dagger)&\leq \Bigg[\sum_{i=1}^N\|F_i^{\delta}\|_*-\|F_i^\dagger\|_*\Bigg]+\|p\|_{\mathcal{H}}\|\Phi F^{\delta}-\Phi F^\dagger\|_{\mathcal{H}}\\
    &\leq \Bigg[\sum_{i=1}^N\|F_i^{\delta}\|_*-\|F_i^\dagger\|_*\Bigg]w+\|p\|_{\mathcal{H}}(\|\Phi F^{\delta}-\m^{\delta}\|_{\mathcal{H}}+\delta).
    \end{aligned}
\end{equation*}
Using \eqref{ineq_optimality} we obtain
\begin{equation}
    \begin{aligned}
    D_H(F^{\delta},F^\dagger)\leq \frac{1}{2\lambda}(\delta^2 - \|\Phi F^{\delta}-\m^{\delta}\|_{\mathcal{H}}^2)+\|p\|_{\mathcal{H}}(\|\Phi F^\delta-\m^\delta\|_{\mathcal{H}}+\delta)
    \end{aligned}.
    \label{aux_bound_bregman}
\end{equation}
Now, we use $ab\leq (a^2+b^2)/2$ with $a=\sqrt{\lambda}\|p\|_{\mathcal{H}}$ and $b=\|\Phi F^{\delta}-\m^\delta\|_{\mathcal{H}}/\sqrt{\lambda}$ to get
\begin{equation*}
    \|p\|_{\mathcal{H}}\|\Phi F^{\delta}-\m^\delta\|_{\mathcal{H}}\leq \frac{\lambda}{2}\|p\|_{\mathcal{H}}^2+\frac{1}{2\lambda}\|\Phi F^\delta-\m^\delta\|_{\mathcal{H}}^2.
\end{equation*}
Injecting this in \eqref{aux_bound_bregman}, we finally obtain
\begin{equation*}
    D_H(F^\delta,F^\dagger)\leq \frac{1}{2\lambda}\delta^2+\frac{\lambda}{2}\|p\|_{\mathcal{H}}^2+\|p\|_{\mathcal{H}}\delta=\frac{1}{2\lambda}(\delta + \lambda\|p\|_{\mathcal{H}})^2=\frac{1}{2 c}(1+c\|p\|_{\mathcal{H}})^2\delta,
\end{equation*}
which yields the first inequality.

To obtain the second inequality, we notice that the non-negativity of $D_H(F^\delta,F)$ in \eqref{aux_bound_bregman} yields
\begin{equation*}
    \begin{aligned}
    0&\leq \frac{1}{2\lambda}(\delta^2 - \|\Phi F^{\delta}-\m^{\delta}\|_{\mathcal{H}}^2)+\|p\|_{\mathcal{H}}(\|\Phi F^\delta-\m^\delta\|_{\mathcal{H}}+\delta)\\
        &=(\|\Phi F^\delta-\m^\delta\|_{\mathcal{H}}+\delta)\bigg(\frac{1}{2\lambda}(\delta-\|\Phi F^\delta -\m^\delta\|_{\mathcal{H}})+\|p\|_{\mathcal{H}}\bigg),
    \end{aligned}
\end{equation*}
which gives $\|\Phi F^\delta-\m^\delta\|_{\mathcal{H}}\leq 2\lambda\|p\|_{\mathcal{H}}+\delta$. Since $\|\Phi F^\delta -\Phi F^\dagger\|_{\mathcal{H}}\leq \|\Phi F^\delta-\m^\delta\|_{\mathcal{H}}+\delta$ we can finally conclude.
\end{proof}

\section{Proof of Lemma~\ref{lemma isomorphism B2 and Bochner}}
\label{appendix_proof_isomorphism}
Let $\{e_i\}_{i\in\NN}$ and $\{f_j\}_{j\in\NN}$ be orthonormal bases of $\Hm$ and $\calH$, respectively. We prove that $\{e_i\otimes f_j\}_{i,j\in\NN}$ forms an orthonormal basis of $\calB_2(\Hm;\calH)$ and that $\{\mathcal{J}(e_i\otimes f_j)\}_{i,j\in\NN}$ is an orthonormal basis of $\HmH$.

Let us start with the first part of the proof. Choosing $\{e_i\}_{i\in\NN}$ as the orthonormal basis of $\Hm$ in the definition of the inner product in $\calB_2(\Hm;\calH)$ \eqref{definition inner product B2}, we have
$$
    \langle e_i\otimes f_j,e_h\otimes f_k\rangle=\langle f_j,f_k\rangle_{\calH}\sum_{l\in\NN}\langle e_i,e_l\rangle_{\Hm}\langle e_h,e_l\rangle_{\Hm}=
    \begin{cases}
        1~&\text{if}~(i,j)=(h,k),\\
        0~&\text{otherwise.}
    \end{cases}
$$
Thus $\{e_i\otimes f_j\}_{i,j\in\NN}$ is an orthonormal system in $\calB_2(\Hm;\calH)$. Let us now prove the completeness. Let $G\in\calB_2(\Hm;\calH)$ be such that $\langle G,e_i\otimes f_j\rangle=0$ for every $i,j\in\NN$. Thus, we have
\begin{equation*}
    0=\langle G,e_i\otimes f_j\rangle=\sum_{l\in\NN}\langle e_i,e_l\rangle_{\Hm}\langle Ge_l,f_j\rangle_{\calH}=\langle G e_i,f_j\rangle_\calH,\qquad i,j\in\NN.
\end{equation*}
Since $\{f_j\}_{j\in\NN}$ is an orthonormal basis for $\calH$, we can conclude that $Ge_i=0$ for every $i$. By the completeness of the family $\{e_i\}_{i\in\NN}$ in $\Hm$, we conclude that $G=0$. 

Now let us set $H_{i,j}=\mathcal{J}(e_i\otimes f_j)$. By the definition of the inner product in $\HmH$, we notice that $\{H_{i,j}\}_{i,j\in\NN}$ forms an orthonormal system in $H^m(\Omega;\calH)$. Indeed, we have
\begin{align*}
    \langle H_{i,j},H_{h,k}\rangle_{\HmH}
    &=\langle f_j,f_k\rangle_\calH\sum_{\lvert\alpha\rvert\leq m}\langle\operatorname{D}^\alpha e_i,\operatorname{D}^\alpha e_h\rangle_{\LD}\\
    &=\langle f_j,f_k\rangle_\calH\langle e_i,e_h\rangle_{\Hm}=
    \begin{cases}
        1~&\text{if}~(i,j)=(h,k),\\
        0~&\text{otherwise.}
    \end{cases}
\end{align*}    
Moreover, it turns out that it is also complete. Indeed, let $H\in H^m(\Omega;\calH)$ be such that $H$ is orthogonal to $H_{i,j}$ for every $i,j\in\NN$, then
\begin{equation}\label{eq identity in proof bochner spaces}
    0=\langle H, H_{i,j}\rangle_{H^m(\Omega;\calH)}=\sum\limits_{\lvert\alpha\rvert\leq m}\int_\Omega \operatorname{D}^\alpha e_i(x) \langle \operatorname{D}^\alpha H(x), f_j\rangle_\calH\, dx.
\end{equation}
For every $j\in\NN$ consider the function
$$
H_j(x)=\langle H(x),f_j\rangle_\calH,\qquad \text{a.e.\ }x\in\Omega.
$$
We have that  $H_j\in H^m(\Omega)$ and that $\operatorname{D}^\alpha H_j(x) = \langle \operatorname{D}^\alpha H(x),f_j\rangle_\calH$ a.e.\ $x\in\Omega$. Indeed we have
\begin{equation*}
\begin{aligned}
    \langle\operatorname{D}^\alpha H_j,\varphi\rangle 
    &= (-1)^{\lvert\alpha\rvert}\int_\Omega\langle H(x),f_j\rangle_\calH \operatorname{D}^\alpha \varphi(x)dx \\
    &= \bigg\langle (-1)^{\lvert\alpha\rvert} \int_\Omega H(x) \operatorname{D}^\alpha \varphi(x)dx,f_j\bigg\rangle_\calH \\
    &= \bigg\langle \int_\Omega \operatorname{D}^\alpha H(x)  \varphi(x)dx,f_j\bigg\rangle_\calH \\
    &=\int_\Omega \langle \operatorname{D}^\alpha H(x),f_j \rangle_\calH \varphi(x)dx, \qquad \text{for every}~\varphi\in \mathcal{C}_c^\infty(\Omega),
    \end{aligned}
\end{equation*}
where in the second and in the last equalities we used the Bochner theorem \citet[Section V.5, Corollary 2]{Yosida1980-mb}.
Therefore, from \eqref{eq identity in proof bochner spaces} we obtain
\begin{equation*}
    0=\langle H, H_{i,j}\rangle_{H^m(\Omega;\calH)}=\sum\limits_{\lvert\alpha\rvert\leq m}\int_\Omega \operatorname{D}^\alpha e_i(x) \operatorname{D}^\alpha H_j(x) dx=\langle e_i,H_j\rangle_{H^m(\Omega)}.
\end{equation*}
Using that $\{e_i\}_{i\in\NN}$ is an orthonormal basis for $H^m(\Omega)$, we can conclude $H_j=0$, i.e.\ $\langle H(x),f_j\rangle_\calH =0$ a.e.\ $x\in\Omega$. Now, since $\{f_j\}_{j\in\NN}$ is an orthonormal basis for $\calH$, then we can conclude $H=0$.

\section{Fr\'echet derivative of the forward map for the Calder\'on problem}
\label{appendix_frechet_derivative}
In this section, we compute and prove the injectivity  of the Fréchet derivative of the forward map $\Psi$ associated to the Calder\'on problem, defined by
\begin{equation*}
\begin{aligned}
        \Psi\colon A\subseteq L^\infty(\Omega)&\to \calB(H^{1/2}(\partial\Omega),H^{-1/2}(\partial\Omega))\\
        q&\mapsto \Lambda_q,
\end{aligned}
\end{equation*}
where $A\eqdef\{q\in L^\infty(\Omega): 0~\mathrm{is~not~an~eigenvalue~of~}-\Delta +q\}$ and $\Lambda_q$ denotes the Dirichlet-to-Neumann map defined in \Cref{equation DN map expression}. These are classical results that we decided to include for the sake of completeness. The proofs are adapted from \citet[Sections 2 and 3]{lechleiter2008newton}, where the same results are shown in the case of the diffusion equation $-\operatorname{div}(\sigma\nabla u)=0$ with current-to-voltage measurements (Neumann-to-Dirichlet map).
In \Cref{lemma compactness F'}, we also prove the compactness of $\Psi'[q](h)$ for every $q\in A$ and $h\in L^\infty(\Omega)$.
In the following proofs we denote  the duality pairing between $H^{1/2}(\partial\Omega)$ and $H^{-1/2}(\partial\Omega)$ by $\langle\cdot,\cdot\rangle$, and, for every $\varphi\in H^{1/2}(\partial\Omega)$, we denote by $e_\varphi\in\HU$ the unique weak solution to
\begin{equation*}\left\{
    \begin{aligned}
        -\Delta e_\varphi+qe_\varphi &=0&&\mathrm{in}~\Omega,\\
        e_\varphi&=\varphi&&\mathrm{on}~\partial\Omega.
    \end{aligned}\right.
\end{equation*}

\begin{lemma}\label{lemma expression and injectivity F'}
    For every $q\in A$ and $h\in L^\infty(\Omega)$ it holds that
    \begin{equation}
    \begin{aligned}
        \Psi'[q](h)\colon H^{1/2}(\partial\Omega)&\to H^{-1/2}(\partial\Omega)\\
        f&\mapsto \restriction{\partial_\nu v}{\partial\Omega},
    \end{aligned}
        \label{equation frechet derivative}
    \end{equation}
    where $v\in H^1(\Omega)$ is the unique weak solution to
    \begin{equation*}\left\{
    \begin{aligned}
        -\Delta v+qv &=-hu&&\mathrm{in}~\Omega,\\
        v&=0&&\mathrm{on}~\partial\Omega
    \end{aligned}\right.
\end{equation*}
    and $u\in H^1(\Omega)$ is the unique weak solution to
    \begin{equation}\left\{
    \begin{aligned}
        -\Delta u+qu &=0&&\mathrm{in}~\Omega,\\
        u&=f&&\mathrm{on}~\partial\Omega.
    \end{aligned}\right.
    \label{eq Schrodinger in appendix}
\end{equation}
Moreover, for every $q\in A$, the mapping $\Psi'[q]\colon L^\infty(\Omega)\to \calB(H^{1/2}(\partial\Omega),H^{-1/2}(\partial\Omega))$ is injective.
\end{lemma}
\begin{proof}
    Let $q\in A$ and $h\in L^\infty(\Omega)$ be such that $h+q\in A$. We denote by $u$ (respectively $u'$) the solution to \Cref{eq Schrodinger in appendix} associated to the potential $q$ (respectively $q+h$). As a consequence, it holds that
    \begin{equation*}
       -\Delta (u'-u)+q(u'-u)=-hu'. 
    \end{equation*}
    Thus, since $q,q+h\in A$, we have
     \begin{equation*}
         \|u'\|_{\HU}\lesssim\|f\|_{H^{1/2}(\partial\Omega)},
     \end{equation*}
     and
     \begin{equation*}
         \|u'-u\|_{\HU}\lesssim \|hu'\|_{\LD}.
     \end{equation*}
     As a result, we get
    \begin{equation}\label{inequality for frechet derivative lemma}
        \|u'-u\|_{H^1(\Omega)}\lesssim \|h\|_{L^\infty(\Omega)}\|f\|_{H^{1/2}(\partial\Omega)}.
    \end{equation}
    Let now $\varphi\in H^{1/2}(\partial\Omega)$. Integrating by parts yields
    \begin{equation*}
        \langle \restriction{\partial_\nu v}{\partial\Omega},\varphi\rangle=\int_\Omega\Delta v e_\varphi -\int_\Omega v\Delta e_\varphi=\int_\Omega h u e_\varphi.
    \end{equation*}
    Analogously, we get
    \begin{equation*}
        \langle \restriction{\partial_\nu u'}{\partial\Omega}-\restriction{\partial_\nu u}{\partial\Omega},\varphi\rangle=\int_\Omega\Delta(u'-u) e_\varphi -\int_\Omega (u'-u)\Delta e_\varphi=\int_\Omega h u' e_\varphi.
    \end{equation*}
    Therefore,
    \begin{equation*}
        \langle \restriction{\partial_\nu u'}{\partial\Omega}-\restriction{\partial_\nu u}{\partial\Omega}-\restriction{\partial_\nu v}{\partial\Omega},\varphi\rangle=\int_\Omega h (u'-u) e_\varphi,
    \end{equation*}
    that, by \Cref{inequality for frechet derivative lemma}, yields
    \begin{equation*}
        \|\restriction{\partial_\nu u'}{\partial\Omega}-\restriction{\partial_\nu u}{\partial\Omega}-\restriction{\partial_\nu v}{\partial\Omega}\|_{H^{-1/2}(\partial\Omega)}\lesssim \|h\|^2_{L^\infty(\Omega)}\|f\|_{H^{1/2}(\partial\Omega)},
    \end{equation*}
    where we used $\|e_\varphi\|_{\HU}\lesssim \|\varphi\|_{H^{1/2}(\partial\Omega)}$.
    This latter ensures $\Psi'[q](h)f=\restriction{\partial_\nu v}{\partial\Omega}$ for every $q\in A$ and $h\in L^\infty(\Omega)$.
    Suppose now that $\Psi'[q](h)f=0$ for every $f\in H^{1/2}(\partial\Omega)$ and let $\varphi\in H^{1/2}(\Omega)$. Integrating by parts yields
    \begin{equation*}
        0=\langle \restriction{\partial_\nu v}{\partial\Omega},\varphi\rangle=\int_\Omega \Delta v e_\varphi -\int_\Omega v\Delta e_\varphi=\int_\Omega h u e_\varphi.
    \end{equation*}
    Using the completeness of the family of products $\{ue_\varphi: f,\varphi\in H^{1/2}(\partial\Omega)\}$ (see \citet{sylvester1987global} for $d\ge3$ and \citet{Bukhgeim+2008+19+33} for $d=2$), we can conclude that $h=0$.
\end{proof}

\begin{lemma}\label{lemma compactness F'}
    The mapping $\Psi'[q](h)$ as in \Cref{equation frechet derivative} is a compact operator for every $q\in A$ and $h\in L^\infty(\Omega)$.
\end{lemma}
\begin{proof}
    Let $\{f_n\}_{n\in\NN}$ be a bounded sequence in $H^{1/2}(\partial\Omega)$ and consider $u_n\in\HU$ the solution to $-\Delta u_n+qu_n=0$ such that $\restriction{u_n}{\partial\Omega}=f_n$. In particular, it holds that
    \begin{equation*}
        \|u_n\|_{\HU}\lesssim \|f_n\|_{H^{1/2}(\partial\Omega)}\leq M,\qquad \text{for some}~M>0.
    \end{equation*}
    Therefore, due to the Rellich-Kondrachov theorem, we have that, up to extracting a subsequence, there exists $u\in L^2(\Omega)$ such that $u_n\to u$ in $L^2(\Omega)$ as $n\to+\infty$. Let $v_n,v \in \HU$ be the solutions to
    \begin{equation*}\left\{
    \begin{aligned}
        -\Delta v_n+qv_n &=-hu_n&&\mathrm{in}~\Omega,\\
        v&=0&&\mathrm{on}~\partial\Omega
    \end{aligned}\right.
\end{equation*}
and 
\begin{equation*}\left\{
    \begin{aligned}
        -\Delta v+qv &=-hu&&\mathrm{in}~\Omega,\\
        v&=0&&\mathrm{on}~\partial\Omega.
    \end{aligned}\right.
\end{equation*}
We now conclude the proof by showing $\|\restriction{\partial_\nu v_n}{\partial\Omega}-\restriction{\partial_\nu v}{\partial\Omega}\|_{H^{-1/2}(\partial\Omega)}\to0$ as $n\to+\infty$. Let $\varphi\in H^{1/2}(\partial\Omega)$,
\begin{equation*}
\begin{aligned}
    \langle\restriction{\partial_\nu v_n}{\partial\Omega}-\restriction{\partial_\nu v}{\partial\Omega},\varphi\rangle 
    &=\int_\Omega \Delta (v_n-v) e_\varphi +\int_\Omega \nabla (v_n-v) \nabla e_\varphi\\
    &=\int_\Omega q(v_n-v)e_\varphi +\int_\Omega h(u_n-u)e_\varphi -\int_\Omega (v_n-v)\Delta e_\varphi\\
    &=\int_\Omega h(u_n-u)e_\varphi.
\end{aligned}
\end{equation*}
Thus, we get
\begin{equation*}
    \|\restriction{\partial_\nu v_n}{\partial\Omega}-\restriction{\partial_\nu v}{\partial\Omega}\|_{H^{-1/2}(\partial\Omega)}\lesssim \|h\|_{L^\infty(\Omega)}\|u_n-u\|_{L^2(\Omega)},
\end{equation*}
where we used $\|e_\varphi\|_{\HU}\lesssim \|\varphi\|_{H^{1/2}(\partial\Omega)}$. Since $u_n\to u$ in $L^2(\Omega)$ as $n\to+\infty$, we conclude the proof.
\end{proof}
\end{document}